\documentclass[12pt, reqno]{amsart}
\usepackage{latexsym}
\usepackage{color,dsfont}
\usepackage{xcolor}
\usepackage{amssymb,amsfonts,amsmath,mathrsfs}
\usepackage{cases}
\newcommand{\kommentar}[1]{}

\newcommand{\acom}[1]{{\color{blue}{Alexandra: #1}} }
\newcommand{\hcom}[1]{{\color{red}{Hung: #1}} }

\newcommand{\powerzeta}{ |\zeta(\tfrac12+\alpha+it)|^{-2k}}

\newtheorem{lem}{Lemma}[section]
\newtheorem{prop}[lem]{Proposition}
\newtheorem{thm}[lem]{Theorem}

\newtheorem{conj}[lem]{Conjecture}
\theoremstyle{definition}

\newtheorem{rem}{Remark}

\numberwithin{equation}{section}

\begin{document}

\title{Negative moments of the Riemann zeta-function}

\author{Hung M. Bui and Alexandra Florea}
\address{Department of Mathematics, University of Manchester, Manchester M13 9PL, UK}
\email{hung.bui@manchester.ac.uk}
\address{UC Irvine, Mathematics Department, Rowland Hall, Irvine 92697, USA}
\email{floreaa@uci.edu}

\begin{abstract}
Assuming the Riemann Hypothesis we study negative moments of the Riemann zeta-function and obtain asymptotic formulas in certain ranges of the shift in $\zeta(s)$. For example, integrating $|\zeta(1/2+\alpha+it)|^{-2k}$ with respect to $t$ from $T$ to $2T$, we obtain an asymptotic formula when the shift $\alpha$ is roughly bigger than $\frac{1}{\log T}$ and $k < 1/2$. We also obtain non-trivial upper bounds for much smaller shifts, as long as $\log\frac{1}{\alpha} \ll \log \log T$. This provides partial progress towards a conjecture of Gonek on negative moments of the Riemann zeta-function, and settles the conjecture in certain ranges. As an application, we also obtain an upper bound for the average of the generalized M\"{o}bius function.

\end{abstract}


\allowdisplaybreaks

\maketitle

\section{Introduction}
For $k>0$, the $2k^{\text{th}}$ moment of the Riemann zeta-function is given by
$$ I_k(T) = \int_0^T  |\zeta(\tfrac{1}{2}+it )|^{2k} \, dt.$$
Hardy and Littlewood \cite{HL} computed the second moment, and Ingham \cite{I} computed the fourth moment. It is conjectured that 
\begin{equation}
I_k(T) \sim c_k T (\log T)^{k^2}
\label{asympt}
\end{equation}
for an explicit constant $c_k$, whose value was predicted by Keating and Snaith \cite{KS2} using analogies with random matrix theory. Their conjecture was later refined by Conrey, Farmer, Keating, Rubinstein and Snaith \cite{CFKRS} to include lower order powers of $\log T$, for integer $k$. 
Under the Riemann Hypothesis (RH), Soundararajan \cite{sound_ub} established almost sharp upper bounds for all the positive moments. This result was later improved by Harper \cite{harper}, who obtained upper bounds of the conjectural magnitude as in \eqref{asympt}. There is a wealth of literature on obtaining lower and upper bounds for positive moments of $\zeta(s)$; for (an incomplete) list of results, we refer the reader to \cite{ramachandra1, ramachandra2, hb, maks, rs, BCR, HRS, heap_sound}. 

In this paper, we are interested in studying \textit{negative} moments of the Riemann zeta-function. For $k>0$ and $\alpha>0$, let
$$I_{-k}(\alpha, T) = \frac{1}{T} \int_0^T \powerzeta \, dt.$$
A conjecture due to Gonek \cite{gonek} states the following. 
\begin{conj}[Gonek] \label{gonek_conj}
Let $k>0$. Uniformly for $\frac{1}{\log T} \leq \alpha \leq 1$,
$$ I_{-k}(\alpha,T) \asymp  \Big(\frac{1}{\alpha} \Big)^{k^2}, $$
and uniformly for $0 < \alpha \leq \frac{1}{\log T}$,
$$
I_{-k}(\alpha, T) \asymp 
\begin{cases}
(\log T)^{k^2} & \mbox{ if } k<1/2, \\
\big(\log\frac{e}{\alpha \log T}\big) (\log T)^{k^2} & \mbox{ if } k=1/2. \\
\end{cases}
$$
\end{conj}

Gonek's original conjecture predicted formulas for $k>1/2$ and $\alpha \leq \frac{1}{\log T}$ as well, which seem to be contradicted however by more recent evidence (see the next paragraph). Under RH, Gonek \cite{gonek} also proved lower bounds of the conjectured order of magnitude for all $k>0$ and $\frac{1}{\log T} \leq \alpha \leq 1$, and for $k<1/2$ and $0<\alpha \leq\frac{1}{\log T}.$ 

No other progress has been made towards Gonek's conjecture so far, but more recent random matrix theory computations due to Berry and Keating \cite{bk} and Forrester and Keating \cite{FK} suggest that certain corrections to the above conjecture are due in some ranges. Namely, when $\alpha \leq \frac{1}{\log T}$, random matrix theory computations seem to contradict Gonek's prediction for the negative moments when $k \geq 3/2$. In particular, the work in \cite{bk, FK} suggests certain transition regimes when $k=(2n+1)/2$, for $n$ a positive integer (Gonek's conjecture already captures the first transition at $k=1/2$ featuring a logarithmic correction, and in this case the two conjectures do agree.)

Reinterpreting the random matrix theory computations in \cite{bk}, one would expect that for a shift $0<\alpha \leq \frac{1}{\log T}$ and $j$ a natural number such that $(2j-1)/2 < k < (2j+1)/2$,
\begin{equation}
 I_{-k}(\alpha, T) \asymp (\log T)^{k^2} (\alpha \log T)^{-j(2k-j)},
 \label{rmt}
 \end{equation} while for $k=(2j-1)/2$ and $j$ a natural number, one would expect
$$  I_{-k}(\alpha, T) \asymp \log \Big( \frac{e}{\alpha \log T} \Big) (\log T)^{k^2} (\alpha \log T)^{-j(2k-j)}.$$
We note that the above prediction indeed agrees with Conjecture \ref{gonek_conj} for $k=1/2$ and $\alpha\leq\frac{1}{\log T}$. We remark that one could also predict \eqref{rmt} for integer $k$ using heuristic ideas similar as in \cite{CFKRS}. 

In this paper, we study the negative moments of the Riemann zeta-function. While obtaining lower bounds for the negative moments is a more tractable problem (see the comment after Conjecture \ref{gonek_conj}), no progress has been made so far on obtaining asymptotic formulas or non-trivial upper bounds.
When the shift $\alpha$ is ``big enough'', we obtain upper bounds which are almost sharp according to Conjecture \ref{gonek_conj}, up to some logarithmic factors. We also obtain the first non-trivial upper bounds for the negative moments for a wide range of much smaller shifts $\alpha$ (roughly $\alpha \gg (\log T)^{-O(1)}$); however, the bounds in these cases are far from sharp.

More precisely, we prove the following.
\begin{thm}
\label{mainthm1}
Assume RH. Let $k \geq 1/2, \alpha>0$ and $\varepsilon, \delta>0$, such that $u=\frac{\log\frac{1}{\alpha}}{ \log \log T}\ll 1$. Then
\begin{equation*}
\frac{1}{T} \int_T^{2T} \powerzeta \, dt  \ll   
\end{equation*}
\begin{numcases}{}
(\log \log T)^k(\log T)^{k^2} & if $ \alpha \gg \frac{(\log \log T)^{\frac{4}{k}+\varepsilon}}{( \log T)^{\frac{1}{2k}}}, $ \label{1}\\
\exp \Big( \frac{ (4+\varepsilon) \log T \log \log \log T}{\log \log T} \Big) & if $\frac{1}{(\log T)^{\frac{1}{2k}}} \ll \alpha = o \big( \frac{ (\log \log T)^{\frac{4}{k}+\varepsilon}}{(\log T)^{\frac{1}{2k}}} \big), $ \label{2} \\ 
T^{(1+\delta)(ku-\frac12+k\varepsilon)}  &  if $ \alpha \ll \frac{1}{(\log T)^{\frac{1}{2k}}}$. \label{3}
\end{numcases}
\end{thm}

We also have the following bounds for $k<1/2$. 

\begin{thm}
\label{mainthm2}
Assume RH. Let $k < 1/2, \alpha>0$ and $\varepsilon, \delta>0$, such that $u=\frac{\log\frac{1}{\alpha}}{ \log \log T}\ll 1$. Then
\begin{equation*}
\frac{1}{T} \int_T^{2T} \powerzeta \, dt  \ll   
\end{equation*}
\begin{numcases}{}
(\log \log T)^k\Big( \frac{\log(\alpha \log T)}{\alpha} \Big)^{k^2} & if $\alpha\gg \frac{ \log \log T}{\log T}, $ \label{4} \\
 \exp \Big(C_1 (\log \log T) \Big( \log \frac{\log \log T}{\alpha \log T } \Big)^2 \Big) & if $ \frac{1}{\log T} \ll \alpha = o \big( \frac{ \log \log T}{\log T} \big), $ \label{5'} \\
\exp \Big( C_2 \Big( \frac{1}{\alpha \log T} \Big)^{\frac{2k}{1-2k-k\varepsilon}} \Big) & if $ \frac{1}{(\log T)^{\frac{1}{2k}-\varepsilon}} \ll \alpha = o \big( \frac{1}{\log T} \big), $\label{6}\\
T^{(1+\delta)(ku-\frac12+k\varepsilon)} \exp \Big( O \Big(\frac{\log T \log \log  \log T}{\log \log T} \Big) \Big) &  if $ \alpha=o\big(\frac{1}{(\log T)^{\frac{1}{2k}-\varepsilon}}\big)$, \label{7}
\end{numcases}
for some constants $C_1, C_2>0$. 
\end{thm}

\kommentar{\begin{thm}
\label{mainthm}
Assume RH.
Let $k>0$, $\alpha>0$ and $\varepsilon>0$ such that $u=\frac{\log\frac{1}{\alpha}}{ \log \log T}\ll 1$. Then
\begin{equation*}
\frac{1}{T} \int_T^{2T} \powerzeta \, dt  \ll   
\end{equation*}
\begin{numcases}{}
(\log \log T)^k\Big( \frac{\log(\alpha \log T)}{\alpha} \Big)^{k^2} & if $\alpha\gg \frac{ \log \log T}{\log T}, k  < 1/2$, \label{first} \\
 \exp \Big(C_1 (\log \log T) \Big( \log \frac{1}{\alpha \log T \log \log T} \Big)^2 \Big) & if $ \frac{1}{\log T} \ll \alpha = o \Big( \frac{ \log \log T}{\log T} \Big), k <1/2$, \label{second} \\
\exp \Big( C_2 \Big( \frac{1}{\alpha \log T} \Big)^{\frac{2k(1+\varepsilon)}{1-2k}} \Big) & if $ \frac{1}{(\log T)^{\frac{1}{2k}-\varepsilon}} \ll \alpha = o \Big( \frac{1}{\log T} \Big), k < 1/2$, \label{third}\\
(\log \log T)^k\Big( \frac{\log(\alpha \log T)}{\alpha} \Big)^{k^2} & if $ \alpha \gg \frac{(\log \log T)^{\frac{4}{k}+\varepsilon}}{( \log T)^{\frac{1}{2k}}}, k \geq 1/2$, \label{fourth}\\
\exp \Big( \frac{ (4+\varepsilon) \log T \log \log \log T}{\log \log T} \Big) & if $\frac{1}{(\log T)^{\frac{1}{2k}}} \ll \alpha = o \Big( \frac{ (\log \log T)^{\frac{4}{k}+\varepsilon}}{(\log T)^{\frac{1}{2k}}} \Big), k \geq 1/2$, \label{fifth} \\ 
T^{(1+\varepsilon)(ku-1/2+2k\varepsilon)}  &  if $ \alpha=o\Big(\frac{1}{(\log T)^{\frac{1}{2k}-\varepsilon}}\Big)$, \label{sixth}
\end{numcases}
for some constants $C_1, C_2>0$. 
\end{thm}
}

\noindent\textbf{Remarks.}
\begin{enumerate}
\item We remark that results about negative moments of quadratic Dirichlet $L$--functions were recently proven in the function field setting, first in \cite{ratios} for shifts big enough, and later improved in \cite{fl_neg} to include small shifts.
\item We note that the  bound \eqref{1} in Theorem \ref{mainthm1} and the bound \eqref{4} in Theorem \ref{mainthm2} above are ``almost sharp''. The former is off possibly by some powers of $\log$, while the latter is off by some $\log \log$ factor.
\end{enumerate}

We also further use the upper bounds to obtain an asymptotic formula for the negative moments in the following ranges.


\begin{thm}
\label{thm_asymptotic}
Assume RH. Let $k>0$, $C,\varepsilon>0$ and $\alpha \geq \max \Big\{ C\frac{(\log \log T)^{\frac{4}{k}+\varepsilon}}{(\log T)^{\frac{1}{2k}}},  \frac{(1+\varepsilon)\log \log T}{2\log T} \Big\}$. Then we have 
$$ \frac{1}{T} \int_T^{2T} \powerzeta \, dt  = \big(1+o(1)\big)\zeta(1+2\alpha)^{k^2}\prod_{p}\Big(1-\frac{1}{p^{1+2\alpha}}\Big)^{k^2}\Big(1+\sum_{j=1}^\infty\frac{\mu_k(p^j)^2}{p^{(1+2\alpha)j}}\Big),$$
where $\mu_k(n)$ denotes the $n^{\text{th}}$ Dirichlet coefficient of $\zeta(s)^{-k}$.
\end{thm}
As an application, we study averages of the generalized M\"{o}bius function. Under RH, Littlewood \cite{L} proved that $\sum_{n \leq x} \mu(n) \ll x^{\frac12+\varepsilon}$. This was improved in several works [\ref{landau}, \ref{T}, \ref{MM}]. Soundararajan \cite{S} showed that  
$$ \sum_{n \leq x} \mu(n) \ll \sqrt{x} \exp \Big(  \sqrt{\log x} (\log \log x)^{14} \Big)$$
on RH. A note of Balazard and de Roton \cite{BR} on Soundararajan's paper \cite{S} improves the power of $\log \log x$ in the bound above, and it is stated in \cite{BR} that the power $(\log \log x)^{\frac52+\varepsilon}$ is the limitation of the ideas in \cite{S, BR}.

Assuming RH and the conjectured order of magnitude of the negative second moment of $\zeta'(\rho)$ by Gonek \cite{gonek} and Hejhal \cite{H}, Ng \cite{ng} also showed that
$$ \sum_{n \leq x} \mu(n) \ll \sqrt{x} (\log x)^{\frac32}.$$

Using Theorem \ref{mainthm1}, we improve the result in \cite{S}. As before, let $\mu_k(n)$ denote the $n^{\text{th}}$ Dirichlet coefficient of $\zeta(s)^{-k}$.  We prove the following.

\begin{thm}\label{mobius}
Assume RH. We have
$$\sum_{n \leq x} \mu_k(n)  \ll \begin{cases}
\sqrt{x} \exp \big(\varepsilon\sqrt{\log x}\big) & \emph{if }k<1,\\
\sqrt{x} \exp \Big((\log x)^{\frac{k}{k+1}} (\log \log x)^{\frac{7}{k+1}+\varepsilon} \Big) & \emph{if }k\geq 1.
\end{cases} $$
Also, if we assume Conjecture \ref{gonek_conj}, then
$$ \sum_{n \leq x} \mu_k(n) \ll \begin{cases}
\sqrt{x} (\log x)^{\frac{k^2}{4}} & \emph{if }k\leq2,\\
\sqrt{x} \exp \Big((\log x)^{\frac{k-2}{k}} (\log \log x)^{-1+\varepsilon} \Big) & \emph{if }k>2.
\end{cases}$$
\end{thm}

Proving Theorem \ref{mobius} requires good bounds for the negative moments of $\zeta(s)$ when roughly $\alpha \gg \frac{1}{(\log T)^{\frac{1}{2k}}}$. We remark that obtaining the bound \eqref{1} in Theorem \ref{mainthm1} is the most delicate part of the proof and requires a recursive argument which allows us to obtain improved estimates at each step.

The ideas in the proof of Theorems \ref{mainthm1} and \ref{mainthm2} are sieve-theoretic inspired ideas, as in the work of Soundararajan \cite{sound_ub} and Harper \cite{harper}. However, unlike in the work in \cite{sound_ub} and \cite{harper}, the contributions of zeros of $\zeta(s)$ play an important part, and one needs to be careful about the choice of parameters in the sieve-theoretic argument, in order to account for the big contributions coming from the zeros.

The paper is organized as follows. In Section \ref{prelim}, we prove some lemmas providing a lower bound for the logarithm of $\zeta(s)$, as well as a pointwise lower bound for $|\zeta(s)|$. In Section \ref{setup}, we introduce the setup of the problem, and prove three main propositions. In Section \ref{proof}, we consider the case of a ``big'' shift $\alpha$, and do the first steps towards proving Theorems \ref{mainthm1} and \ref{mainthm2} in this case. In Section \ref{bigshift}, we obtain the bound \eqref{1} in the smaller region $\alpha \gg \frac{1}{(\log T)^{\frac{1}{2k}-\varepsilon}}$, and prove the bounds \eqref{4}, \eqref{5'} and \eqref{6} from Theorem \ref{mainthm2} as well. In Section \ref{s_recursive}, we obtain the bound \eqref{1} in the wider region stated in the theorem by using a recursive argument. We consider the cases of ``small" shifts in Section \ref{smallshift} and prove the bounds \eqref{3} and \eqref{7}. We then prove Theorem \ref{thm_asymptotic} in Section \ref{section_formula}, and Theorem \ref{mobius} in Section \ref{section_mobius}.

Throughout the paper $\varepsilon$ denotes an arbitrarily small positive number whose value may change from one line to the next.\\

 \textsc{Acknowledgments.} The authors would like to thank Steve Gonek, Jon Keating and Nathan Ng for helpful conversations and comments. The second author gratefully acknowledges support from NSF grant DMS-2101769 while working on this paper. 
\section{Preliminary lemmas}
\label{prelim}
The first two lemmas concern the lower bounds for $|\zeta(s)|$.
\begin{lem}\label{key_ineq}
Assume RH. Let $\alpha >0$. Then
\begin{align*} &\log | \zeta(\tfrac12+\alpha+it) |\geq \frac{ \log  \frac{ t}{2 \pi } }{2 \pi \Delta} \log \big(1- e^{-2 \pi \Delta\alpha} \big)+\Re\Big( \sum_{p \leq e^{2 \pi \Delta}} \frac{(\log p)a_{\alpha}(p;\Delta)}{p^{1/2+\alpha+it}}\Big)\\
	&\qquad\qquad - \frac{\log \log \log T}{2} - \Big( \log \frac{1}{\Delta \alpha} \Big)^{\gamma(\Delta)} + O \Big(  \frac{\Delta^2 e^{\pi \Delta}}{1+\Delta t} + \frac{\Delta \log (1+ \Delta \sqrt{t})}{\sqrt{t}}+1\Big),
\end{align*}
where \begin{equation}\label{formulaa}
a_{\alpha}(n;\Delta)=\sum_{j=0}^{\infty} \Big( \frac{(j+1)}{ \log n+2 \pi j \Delta} e^{-2 \pi j\Delta\alpha}- \frac{(j+1)n^{2\alpha}}{ 2\pi(j+2)\Delta - \log n} e^{-2 \pi(j+2) \Delta\alpha}  \Big)\end{equation}
and $$\gamma(\Delta)=\begin{cases}1 & \text{if }\Delta \alpha = o(1),\\ 0& \text{if }\Delta \alpha \gg 1.\end{cases}$$
\end{lem}
\begin{proof}
We use the work in \cite{CC}. Let
$$f_{\alpha}(x) = \log \Big( \frac{4+x^2}{\alpha^2+x^2} \Big),$$ and let $m_{\Delta}(x)$ be an extremal majorant for $f_{\alpha}$, whose Fourier transform $\widehat{m}_{\Delta}$ is supported in $[-\Delta, \Delta]$, satisfying the properties from Lemma 8 in \cite{CC}. Equation $3.1$ in \cite{CC} gives
\begin{equation} \log | \zeta(\tfrac12+\alpha+it) |  \geq \Big( 1-\frac{\alpha}{2} \Big) \log \frac{t}{2} - \frac{1}{2} \sum_{\gamma} m_{\Delta}(t-\gamma)+O(1),
\label{ineq1}
\end{equation}
where the sum over $\gamma$ is over the ordinates of the zeros of $\zeta(s)$ on the critical line. Using the explicit formula, Equation $3.2$ in \cite{CC} leads to
\begin{align}
\sum_{\gamma} m_{\Delta}(t-\gamma) &=  m_{\Delta} \Big(t - \frac{1}{2i}   \Big) +m_{\Delta} \Big(t + \frac{1}{2i}   \Big)  - \frac{\log \pi}{2 \pi } \widehat{m}_{\Delta}(0)  \label{explicit} \\
&\ + \frac{1}{2 \pi } \int_{-\infty}^{\infty} m_{\Delta}(x) \Re \bigg(\frac{\Gamma'}{\Gamma} \Big( \frac{1}{4} + \frac{i (t-x)}{2} \Big) \bigg)\, dx - \frac{1}{\pi } \Re\bigg(\sum_{n=2}^{\infty} \frac{\Lambda(n)}{n^{1/2+it}} \widehat{m}_{\Delta} \Big( \frac{\log n}{2 \pi}  \Big)\bigg),\nonumber 
\end{align}
where 
\begin{align*}
\widehat{m}_{\Delta}(\xi) &= \sum_{j=0}^{\infty}\bigg(  \frac{j+1}{\xi+j\Delta} \Big(  e^{-2 \pi ( \xi+j\Delta)\alpha}- e^{-4 \pi (\xi+j\Delta)} \Big) \\
&\qquad\qquad\qquad\qquad- \frac{j+1}{(j+2)\Delta-\xi} \Big(e^{-2 \pi ((j+2)\Delta-\xi)\alpha} - e^{-4 \pi ((j+2)\Delta-\xi)} \Big) \bigg).
\end{align*}
For $n \leq e^{2 \pi \Delta}$ we note that 
$$ \widehat{m}_{\Delta} \Big( \frac{\log n}{2 \pi } \Big)= \frac{2\pi a_{\alpha}(n;\Delta)}{n^\alpha} + O\Big(\frac{1}{n^{2}\log n}\Big).$$
Combining that with \eqref{ineq1}, \eqref{explicit} and Equations $(3.3), (3.4), (3.5)$ in \cite{CC} we obtain
\begin{align}\label{allterms} 
\log | \zeta(\tfrac12+\alpha+it) |& \geq \frac{ \log  \frac{ t}{2 \pi } }{2 \pi \Delta} \log \big(1- e^{-2 \pi \Delta\alpha} \big)+\Re\Big( \sum_{n \leq e^{2 \pi \Delta}} \frac{\Lambda(n)a_{\alpha}(n;\Delta)}{n^{1/2+\alpha+it}} \Big) \\
	&\qquad\qquad+ O \Big(  \frac{\Delta^2 e^{\pi \Delta}}{1+\Delta t} + \frac{\Delta \log (1+ \Delta \sqrt{t})}{\sqrt{t}}+1\Big).\nonumber
\end{align}

It is easy to see that $a_{\alpha}(n;\Delta)\geq0$. Also, with $\log n\leq 2\pi\Delta$ we have 
\kommentar{\begin{align}
a_{\alpha}(n;\Delta)&=  \sum_{j=0}^{\infty}(j+1) e^{-2 \pi j \Delta\alpha}\Big(  \frac{1}{\log n + 2 \pi j \Delta}  - \frac{1}{2 \pi(j+2) \Delta-\log n}\Big)\\
&\qquad+\sum_{j=0}^{\infty}\frac{(j+1)e^{-2 \pi j \Delta\alpha}}{2 \pi(j+2) \Delta-\log n}  \big(1-n^{2\alpha}e^{-4 \pi \Delta\alpha}  \big) \label{sumj} \\
&\leq 2(2 \pi \Delta - \log n) \sum_{j=0}^{\infty} \frac{(j+1) e^{-2 \pi j \Delta\alpha}}{(\log n+2 \pi j \Delta)(2 \pi(j+2) \Delta-\log n)}+\frac{1}{2\pi\Delta}\frac{1-n^{2\alpha}e^{-4 \pi \Delta\alpha}}{1-e^{-2 \pi \Delta\alpha}} \\
& \leq \frac{1}{\log n} + \frac{1}{\pi \Delta} \log \frac{1}{1-e^{-2\pi \Delta \alpha}} + O \Big( \frac{1}{\Delta} \Big) \leq \frac{1}{\log n} + \frac{1}{\pi \Delta} \Big( \log \frac{1}{\Delta \alpha} \Big)^{\gamma(\Delta)}+ O \Big( \frac{1}{\Delta} \Big),
\end{align}
where $\gamma(\Delta)=1$ if $\Delta \alpha = o(1)$ and $\gamma(\Delta)=1$ if $\Delta \alpha \gg 1$.
}

\begin{align}
a_{\alpha}(n;\Delta)&<\frac{1}{\log n}+  \sum_{j=1}^{\infty}(j+1) e^{-2 \pi j \Delta\alpha}\Big(  \frac{1}{\log n + 2 \pi j \Delta}  - \frac{1}{2 \pi(j+2) \Delta-\log n}\Big) \nonumber\\
&\qquad\qquad+\sum_{j=1}^{\infty}\frac{(j+1)e^{-2 \pi j \Delta\alpha}}{2 \pi(j+2) \Delta-\log n}  \big(1-n^{2\alpha}e^{-4 \pi \Delta\alpha}  \big) \nonumber \\
&\leq \frac{1}{\log n}+2(2 \pi \Delta - \log n) \sum_{j=1}^{\infty} \frac{(j+1) e^{-2 \pi j \Delta\alpha}}{(\log n+2 \pi j \Delta)(2 \pi(j+2) \Delta-\log n)}\label{sumj'}\\
&\qquad\qquad+\frac{1-n^{2\alpha}e^{-4 \pi \Delta\alpha}}{2\pi\Delta}\sum_{j=1}^{\infty}e^{-2 \pi j \Delta\alpha}\nonumber \\
&< \frac{1}{\log n}+\frac{2(2 \pi \Delta - \log n)}{(2\pi\Delta)^2} \sum_{j=1}^{\infty} \frac{e^{-2 \pi j \Delta\alpha}}{j}+\frac{1-n^{2\alpha}e^{-4 \pi \Delta\alpha}}{2\pi\Delta(1-e^{-2 \pi \Delta\alpha})}\nonumber \\
& = \frac{1}{\log n} -\frac{2(2 \pi \Delta - \log n)}{(2\pi\Delta)^2} \log\big(1-e^{-2 \pi \Delta\alpha}\big)+O \Big( \frac{1}{\Delta} \Big).\label{sumj}
\end{align}
Note that if $\Delta \alpha \gg 1$, then the second term above is $O(\Delta^{-1})$, and if $\Delta \alpha = o(1),$ then it is
\begin{equation}
\leq \frac{1}{\pi \Delta} \log \frac{1}{\Delta \alpha}.
\label{bound_s}
\end{equation}

Now we use the expression \eqref{sumj} for $a_{\alpha}(n;\Delta)$ and evaluate the contribution from the prime squares in \eqref{allterms}.
When dealing with the main term of size $1/\log n$ above, we proceed similarly as in \cite{harper} or Lemma $2$ in \cite{S}. For the second term, we trivially bound the sum over primes and use the bound \eqref{bound_s}, and it follows that the contribution from prime squares is  
\begin{align*}
\geq - \frac{ \log \log \log T}{2} -  \Big(  \log \frac{1}{\Delta \alpha}\Big)^{\gamma(\Delta)} -O(1).
\end{align*}
Dealing similarly with the sum over prime cubes and higher powers and combining the equation above and \eqref{allterms} finishes the proof.  
\end{proof}

\begin{lem}\label{CClemma}
Assume RH. If $0<\alpha =o(\frac{1}{\log\log t})$, then 
\begin{align}\label{newCCbound} \log | \zeta(\tfrac12+\alpha+it) |&\geq \frac{ \log  t }{2\log\log t} \log \big(1-(\log t)^{-2\alpha} \big)\nonumber\\
&\qquad\qquad+ O \Big(  \frac{ \log  t }{\log\log t} \Big)+ O \Big(  \frac{ \log  t }{(\log\log t)^2} \log\frac{1}{1-(\log t)^{-2\alpha}}\Big),
\end{align}
and if $1/2-\alpha\ll\frac{1}{\log\log t}$, then
\begin{align*} \log | \zeta(\tfrac12+\alpha+it) |\geq-\log\log\log t+O(1),
\end{align*}
otherwise
\begin{align*} \log | \zeta(\tfrac12+\alpha+it) |\geq-\Big(\frac12+\frac{8\alpha}{1-4\alpha^2}\Big)\frac{(\log t)^{1-2\alpha}}{\log\log t}-\log\log\log t+O\Big(\frac{(\log t)^{1-2\alpha}}{(1-2\alpha)^2(\log\log t)^2}\Big).
\end{align*}
\end{lem}
\begin{proof}
Carneiro and Chandee  established the last two bounds in \cite[Theorem 2]{CC} and in the case $0<\alpha =o(\frac{1}{\log\log t})$ they obtained
\[
\log | \zeta(\tfrac12+\alpha+it) |\geq \frac{ \log  t }{2\log\log t} \log \big(1-(\log t)^{-2\alpha} \big)+ O \Big(  \frac{( \log  t)^{1-2\alpha} }{(\log\log t)^2(1-(\log t)^{-2\alpha})} \Big).
\]
We will now prove the improved bound \eqref{newCCbound}.

Let $\Delta=\frac{\log\log t}{\pi}$. From \eqref{allterms} and \eqref{sumj} we have
\begin{align*} 
&\log | \zeta(\tfrac12+\alpha+it) | \geq \frac{ \log  \frac{ t}{2 \pi } }{2 \log\log t} \log \big(1- (\log t)^{-2\alpha} \big)\\
	&\qquad- \sum_{n \leq (\log t)^2} \frac{\Lambda(n)}{n^{1/2+\alpha}}\bigg(\frac{1}{\log n} -\frac{2(2\log\log t - \log n)}{(2\log\log t)^2} \log\big(1-(\log t)^{-2\alpha}\big)+O \Big( \frac{1}{\log\log t} \Big)\bigg)\\
&\qquad+O(1).\nonumber
\end{align*}
We now consider the sum over $n$ above, which expands into three terms. The contribution of the first and the last term is
\[
\ll\frac{\log t}{\log\log t},
\]
while the second term is
\begin{align*}
&=-\frac{\log(1-(\log t)^{-2\alpha} )}{\log\log t}\Big(\frac{2(\log t)^{1-2\alpha}}{1-2\alpha}+O\big((\log\log t)^3\big)\Big)\\
&\qquad\qquad+\frac{\log(1-(\log t)^{-2\alpha} )}{2(\log\log t)^2}\Big(\frac{4(\log t)^{1-2\alpha}\log\log t}{1-2\alpha}+O(\log t)\Big)\\
&=O\Big(  \frac{ \log  t }{(\log\log t)^2} \log\frac{1}{1-(\log t)^{-2\alpha}}\Big),
\end{align*}
by the prime number theorem and partial summation. This establishes \eqref{newCCbound}.
\end{proof}

 The next lemma is the classical mean value theorem (see, for instance, \cite[Theorem 6.1]{M}).
\begin{lem}\label{harperresultreplace}
			For any complex numbers $a(n)$ we have
			\begin{equation*}
				\int_T^{2T} \Big|\sum_{n\leq N} \frac{a(n)}{n^{1/2+it}}\Big|^{2} \, dt = \big(T+O(N)\big)\sum_{n\leq N}\frac{|a(n)|^2}{n}.
			\end{equation*}	\end{lem}

Now let $\ell$ be an integer, and for $z \in \mathbb{C}$, let
$$E_{\ell}(z) = \sum_{s \leq \ell} \frac{z^s}{s!}.$$
If $z \in \mathbb{R}$, then for $\ell$ even, $E_{\ell}(z) >0$. 

We have the following elementary inequality.
\begin{lem}
\label{elem_ineq}
Let $\ell$ be an even integer and let $z$ be a complex number such that $|z| \leq \ell/e^2$. Then
$$e^{\Re(z)} \leq \max \Big\{ 1, |E_{\ell}(z)| \Big(1+\frac{1}{15e^{\ell}}\Big)\Big\}.$$
\end{lem}	
\begin{proof}
We have 
\begin{align*}
e^{\Re (z)} = |e^z| \leq |e^z - E_{\ell}(z)|+|E_{\ell}(z)|.
\end{align*}
Now we have
$$|e^z-  E_{\ell}(z)| \leq \sum_{j=\ell+1}^{\infty} \frac{|z|^j}{j!},$$ and we can proceed as in \cite{sound_maks} (see Lemma $1$) to get that
$$ |e^z-  E_{\ell}(z)| \leq \frac{1}{16e^{\ell}}.$$ Hence
$$ e^{\Re(z)} \leq |E_{\ell}(z)| + \frac{1}{16e^{\ell}}.$$ 
If $|E_{\ell}(z)| \leq \frac{15}{16}$, then we have $e^{\Re(z)} \leq 1$. Otherwise, if $|E_{\ell}(z)| > \frac{15}{16}$, we get that
\begin{equation*}
e^{\Re(z) } \leq |E_{\ell}(z)| \Big( 1+\frac{1}{15 e^{\ell}} \Big),
\end{equation*}
and this completes the proof.
\end{proof}

Let $\nu(n)$ be the multiplicative function given by
$$\nu(p^j) = \frac{1}{j!},$$ for $p$ a prime. We will frequently use the following fact. For any interval $I$, $s \in \mathbb{N}$ and $a(n)$ a completely multiplicative function, we have
\begin{equation}
 \Big(  \sum_{p \in I}  a(p) \Big)^s= s!  \sum_{\substack{ p |n \Rightarrow p \in I \\ \Omega(n) = s}} a(n) \nu(n),
 \label{identity}
 \end{equation} where $\Omega(n)$ denotes the number of prime factors of $n$, counting multiplicity.

\section{Setup of the proof and main propositions}
\label{setup}
We will first introduce some notation and state some key lemmas, before proceeding to the proof of Theorems \ref{mainthm1} and \ref{mainthm2}.

Let $$I_0 = (1, T^{\beta_0}],\ I_1 = (T^{\beta_0}, T^{\beta_1}],\ \ldots,\ I_K = (T^{\beta_{K-1}}, T^{\beta_K}]$$ for a sequence $\beta_0 , \ldots, \beta_K$ to be chosen later such that $\beta_{j+1}=r\beta_j$, for some $r>1$.

Also let $\ell_j$ be even parameters which we will also choose later on. Let $s_j$ be integers. For now, we can think of $s_j \beta_j \asymp 1$, and $\sum_{h=0}^K \ell_h \beta_h \ll 1$.

We let $T^{\beta_j} = e^{2 \pi \Delta_j}$ for every $0\leq j \leq K$ and let 
$$P_{u,v}(t)=\sum_{p\in I_u} \frac{b_{\alpha}(p;\Delta_v)}{p^{1/2+\alpha+it}},$$ where
$b_{\alpha}(n;\Delta)$ is a completely multiplicative function in the first variable and
\begin{align*}
b_{\alpha}(p;\Delta) =-a_{\alpha}(p;\Delta)\log p.
\end{align*}

For $p \leq e^{2 \pi \Delta}$, if $\Delta\alpha\gg 1$, we note from equation \eqref{sumj'} that 
\begin{align*}
|b_{\alpha}(p;\Delta) | &\leq1+ \frac{2 \pi \Delta - \log p}{2\pi\Delta} \sum_{j=1}^{\infty} \frac{2\log p }{\log p+2 \pi j \Delta}e^{-2 \pi j \Delta\alpha}+\frac{\log p}{2\pi\Delta}\sum_{j=1}^{\infty}e^{-2 \pi j \Delta\alpha}\\
& \leq   1+\sum_{j=1}^{\infty} e^{-2\pi j\Delta \alpha} = \frac{1}{1-e^{-2\pi \Delta \alpha}}.
\end{align*}
On the other hand, if $\Delta\alpha = o(1)$, then from \eqref{sumj} we get that
\begin{align*}
|b_{\alpha}(p;\Delta)| \leq \frac{2  (2\pi\Delta-\log p)\log p}{(2 \pi \Delta)^2} \sum_{j=1}^{\infty} \frac{e^{-2\pi j \Delta \alpha}}{j }+O(1) \leq \frac{1}{2} \log \frac{1}{\Delta \alpha} +O(1).
\end{align*}
%
Thus we can rewrite the bound for $b_{\alpha}(p;\Delta)$ as a single bound,
\begin{equation}
\label{improved_bd}
|b_{\alpha}(p;\Delta)| \leq b(\Delta) \Big( \log \frac{1}{\Delta \alpha} \Big)^{\gamma(\Delta)},
\end{equation} 
where $\gamma(\Delta)=0$ if $\Delta \alpha \gg 1$ and $\gamma(\Delta)=1$ if $\Delta \alpha=o(1)$, and where
\begin{equation}
\label{bn}
b(\Delta) = 
\begin{cases}
 \frac{1}{1-e^{-2\pi \Delta \alpha}} & \mbox{ if } \Delta \alpha \gg 1,\\
\frac12+\varepsilon & \mbox{ if } \Delta \alpha =o(1).
\end{cases}
\end{equation}

Let 
$$ \mathcal{T}_u = \Big\{ T \leq t \leq 2T : \max_{u \leq v \leq K} |P_{u,v}(t)| \leq \frac{\ell_u}{ke^2}\Big\}.$$ 
Denote the set of $t$ such that $t \in \mathcal{T}_u$ for all $u \leq K$ by $\mathcal{T}'$. For $0 \leq j \leq K-1$, let $\mathcal{S}_j$ denote the subset of $t \in [T, 2T]$ such that $t \in \mathcal{T}_h$ for all $h \leq j$, but $t \notin \mathcal{T}_{j+1}$. 

We will prove the following lemma.
\begin{lem}
\label{lem_initial}
For $t \in [T, 2T]$, we either have 
$$ \max_{0 \leq v \leq K} |P_{0,v}(t)| > \frac{\ell_0}{ke^2},$$ or
$$|\zeta(\tfrac12+\alpha+it)|^{-2k} \leq S_1(t) +S_2(t),$$ where 
\begin{align*}
&S_1(t) =(\log \log T)^k\Big( \frac{1}{1-T^{-\beta_K\alpha}}\Big)^{ \frac{2k}{\beta_K } } \exp \Big( 2k \Big( \log \frac{1}{\Delta_K \alpha} \Big)^{\gamma(\Delta_K)} \Big)  \\
&\  \times  \prod_{h=0}^K \max \Big\{1, |E_{\ell_h}(kP_{h,K}(t))|^2\Big(1+\frac{1}{15e^{\ell_h}}\Big)^2 \Big\} \exp \Big( O \Big(  \frac{\Delta_K^2 e^{\pi \Delta_K}}{1+\Delta_K t} + \frac{\Delta_K \log (1+ \Delta_K \sqrt{t})}{\sqrt{t}}+1\Big)  \Big)
\end{align*}
and
\begin{align*}
S_2(t) &=(\log \log T)^k \sum_{j=0}^{K-1}  \sum_{v=j+1}^K  \Big( \frac{1}{1-T^{-\beta_j\alpha}}\Big)^{ \frac{2k}{\beta_j } } \exp \Big( 2k \Big( \log \frac{1}{\Delta_j \alpha} \Big)^{\gamma(\Delta_j)} \Big)  \\
&\qquad \times \prod_{h=0}^j\max \Big\{1, |E_{\ell_h}(kP_{h,j}(t))|^2\Big(1+\frac{1}{15e^{\ell_h}}\Big)^2 \Big\}  \Big( \frac{ke^2}{\ell_{j+1}} |P_{j+1,v}(t)|\Big)^{2s_{j+1}} \\
&\qquad \times  \exp \Big( O \Big(  \frac{\Delta_j^2 e^{\pi \Delta_j}}{1+\Delta_j t} + \frac{\Delta_j \log (1+ \Delta_j \sqrt{t})}{\sqrt{t}}+1\Big)  \Big).
\end{align*}
\end{lem}
\begin{proof}
For $T \leq t \leq 2T$, we have the following possibilities:
\begin{enumerate}
\item $t \notin \mathcal{T}_0$;
\item $t \in \mathcal{T}'$;
\item $t \in \mathcal{S}_j$  for some $0 \leq j \leq K-1.$
\end{enumerate} 

If $t \notin \mathcal{T}_0$, then the first condition is automatically satisfied. Now suppose that $t \in \mathcal{T}_0$. 

Suppose that $t \in \mathcal{T}'$. We use Lemma \ref{key_ineq} with $\Delta=\Delta_K$ and the inequality in Lemma \ref{elem_ineq} with $z = k\Re P_{h,K}(t)$. 

Now if $t \in \mathcal{S}_j$, then we use Lemma \ref{key_ineq} with $\Delta=\Delta_j$, the inequality in Lemma \ref{elem_ineq} with $z = k \Re P_{h,j}(t)$ and the fact that there exists some $v \geq j+1$ such that $\frac{ke^2}{\ell_{j+1}} |P_{j+1,v}(t)|>1$.
\end{proof}

We will also need the following propositions.
\begin{prop}
\label{contrib_0}
For $0 \leq v \leq K$ and $\beta_0 s_0 \leq 1$, we have
$$ \int_T^{2T} |P_{0,v}(t)|^{2s_0} \, dt \ll Ts_0! b(\Delta_0)^{2s_0} \Big(\log \frac{1}{\Delta_0\alpha}\Big)^{2s_0\gamma(\Delta_0)} (\log \log T^{\beta_0})^{s_0}.$$
\end{prop}
\begin{proof}
Using \eqref{identity}, we have
\begin{align*}
  P_{0,v}(t)^{s_0} = s_0! \sum_{\substack{p|n \Rightarrow p \leq T^{\beta_0} \\ \Omega(n) =s_0}} \frac{b_{\alpha}(n;\Delta_v) \nu(n)}{n^{1/2+\alpha+it}}.
\end{align*}
It then follows from Lemma \ref{harperresultreplace} that
\begin{align*}
\int_T^{2T} |P_{0,v}(t)|&^{2s_0} \,  dt= \big(T+O(T^{\beta_0s_0})\big) (s_0!)^2 \sum_{\substack{p|n \Rightarrow p \leq T^{\beta_0} \\ \Omega(n)=s_0}} \frac{b_{\alpha}(n;\Delta_v)^2 \nu(n)^2}{n^{1+2\alpha}}\nonumber\\
&\ll T s_0! \bigg( \sum_{p \leq T^{\beta_0}} \frac{b_{\alpha}(p;\Delta_v)^2}{p^{1+2\alpha}} \bigg)^{s_0} \leq Ts_0! b(\Delta_v)^{2s_0} \Big(\log \frac{1}{\Delta_v\alpha}\Big)^{2s_0\gamma(\Delta_v)} (\log \log T^{\beta_0})^{s_0},
\end{align*}
where we have used the fact that $\nu(n)^2 \leq \nu(n)$ and the bound \eqref{improved_bd}. Now since 
$$b(\Delta_v) \Big(\log \frac{1}{\Delta_v\alpha}\Big)^{2s_0\gamma(\Delta_v)} \leq b(\Delta_0)\Big(\log \frac{1}{\Delta_0\alpha}\Big)^{2s_0\gamma(\Delta_0)},$$ the conclusion follows.
\end{proof}

\begin{prop}
\label{j+1}
Let $0 \leq j \leq K-1$. For $\sum_{h=0}^j \ell_h \beta_h+s_{j+1} \beta_{j+1} \leq 1$ and for $j+1\leq v \leq K$, we have
\begin{align*}
 \int_T^{2T} & \prod_{h=0}^j |E_{\ell_h}(k P_{h,j}(t))|^2 |P_{j+1,v}(t)|^{2s_{j+1}} \, dt \ll T s_{j+1}! \big( \log T^{\beta_j} \big)^{k^2 b(\Delta_j)^2 (\log \frac{1}{\Delta_j \alpha})^{2 \gamma(\Delta_j)}} \\
 & \times   b(\Delta_{j+1})^{2s_{j+1}} \Big( \log \frac{1}{\Delta_{j+1} \alpha} \Big)^{2s_{j+1} \gamma(\Delta_{j+1})} (\log r)^{s_{j+1}}.
 \end{align*}
\end{prop}
\begin{proof}
Using \eqref{identity}, we have 
\begin{align}
\label{to_bd}
\int_T^{2T} \prod_{h=0}^j |E_{\ell_h}(k P_{h,j}(t))|^2 |P_{j+1,v}(t)|^{2s_{j+1}} \, dt = (s_{j+1}!)^2 \int_T^{2T} \Big| \sum_{n \leq T^{\sum_{h=0}^j \ell_h \beta_h+s_{j+1}\beta_{j+1}}} \frac{c(n) \nu(n)}{n^{1/2+\alpha+it}} \Big|^2 \, dt,
\end{align}
where
$$c(n) = \sum_{\substack{n = n_1 \ldots n_{j+1} \\ p|n_h \Rightarrow p \in I_h \\ \Omega(n_h) \leq \ell_h, h=0,\ldots,j\\\Omega(n_{j+1})=s_{j+1}}}\Big(\prod_{h=0}^j b_{\alpha}(n_h;\Delta_j)k^{\Omega(n_h)} \Big)b_{\alpha}(n_{j+1};\Delta_v) .$$
Using Lemma \ref{harperresultreplace} in \eqref{to_bd} and the fact that $\nu(m)^2 \leq \nu(m)$ for any $m$, we obtain that
\begin{align*}
\eqref{to_bd}  &\ll \Big(T+T^{\sum_{h=0}^j \ell_h \beta_h+s_{j+1}\beta_{j+1}} \Big) (s_{j+1}!)^2 \Big(\prod_{h=0}^j \sum_{\substack{p|n_h \Rightarrow p \in I_h \\ \Omega(n_h) \leq \ell_h}} \frac{|b_{\alpha}(n_h;\Delta_j)|^2 k^{2 \Omega(n_h)} \nu(n_h)}{n_h^{1+2\alpha}} \Big)\\
&\qquad \times  \sum_{\substack{p|n_{j+1} \Rightarrow p \in I_{j+1} \\ \Omega(n_{j+1})=s_{j+1}}}\frac{|b_{\alpha}(n_{j+1};\Delta_v)|^2  \nu(n_{j+1})}{n_{j+1}^{1+2\alpha}}.
\end{align*}
Now we use the assumption that $\sum_{h=0}^j \ell_h \beta_h +s_{j+1} \beta_{j+1} \leq 1$. Bounding $\nu(n_h) \leq 1$, removing the condition on the number of primes of $n_h$, and using \eqref{improved_bd}, we get that
\begin{align*}
\eqref{to_bd} & \ll T s_{j+1}! \prod_{p \leq T^{\beta_j}} \Big( 1- \frac{k^2 |b_{\alpha}(p;\Delta_{j})|^2}{p^{1+2\alpha}}\Big)^{-1} \Big( \sum_{p \in I_{j+1}} \frac{|b_{\alpha}(p;\Delta_{v}|^2}{p^{1+2\alpha}} \Big)^{s_{j+1}} \\
& \ll T s_{j+1}! \big( \log T^{\beta_j} \big)^{k^2 b(\Delta_j)^2 (\log \frac{1}{\Delta_j \alpha})^{2 \gamma(\Delta_j)}}  b(\Delta_v)^{2s_{j+1}} \Big( \log \frac{1}{\Delta_v \alpha} \Big)^{2s_{j+1} \gamma(\Delta_v)}\Big(\log\frac{\beta_{j+1}}{\beta_j}\Big)^{s_{j+1}} \\
& \ll T s_{j+1}! \big( \log T^{\beta_j} \big)^{k^2 b(\Delta_j)^2 (\log \frac{1}{\Delta_j \alpha})^{2 \gamma(\Delta_j)}}  b(\Delta_{j+1})^{2s_{j+1}} \Big( \log \frac{1}{\Delta_{j+1} \alpha} \Big)^{2s_{j+1}  \gamma(\Delta_{j+1})} (\log r)^{s_{j+1}},
\end{align*}
which finishes the proof of the proposition.
\end{proof}
A minor modification of the proposition above (where we do not have the contribution from the $j+1$ interval) yields the following.

\begin{prop}
\label{K_lemma}
For $\sum_{h=0}^K \ell_h \beta_h \leq 1$, we have
\begin{align*}
 \int_T^{2T} & \prod_{h=0}^K |E_{\ell_h}(k P_{h,K}(t))|^2  \, dt \ll T\big( \log T^{\beta_K} \big)^{k^2 b(\Delta_K)^2 (\log \frac{1}{\Delta_K \alpha})^{2 \gamma(\Delta_K)}} .
 \end{align*}
\end{prop}

\kommentar{\begin{lem}
\label{squares}
The estimates in Propositions \ref{j+1} and \ref{K_lemma} still hold when introducing the contribution from the primes square.
\end{lem}}

\section{The case $\alpha \gg \frac{1}{(\log T)^{\frac{1}{2k}-\varepsilon}}$, first steps}
\label{proof}
Here, we will consider the case $\alpha \gg \frac{1}{(\log T)^{\frac{1}{2k}-\varepsilon}}$, which will be a starting point in the proof of Theorems \ref{mainthm1} and \ref{mainthm2}.


We choose
\begin{equation}
\beta_0 = \frac{a(2d-1) ( \log \log T)^2 }{(1+2\varepsilon)k ( \log T)\big( \log \frac{1}{1-(\log T)^{-2\alpha}} \big)}, \quad s_0 = \Big[\frac{a}{\beta_0}\Big], \quad \ell_0 = 2 \Big\lceil \frac{s_0^d}{2}\Big \rceil \label{beta_0}
\end{equation} 
and
\begin{equation}
\label{beta_j}
\beta_j = r^j \beta_0 , \quad s_j = \Big[\frac{a}{\beta_j}\Big], \quad \ell_j =2\Big \lceil \frac{s_j^d}{2}\Big\rceil, \quad 1 \leq j \leq K,
\end{equation}
where we can pick, for example,
$$ a= \frac{4-3k\varepsilon}{2(2-k\varepsilon)}, \, r = \frac{2}{2-k\varepsilon} , \,
d = \frac{8-7k\varepsilon}{2(4-3k\varepsilon)},$$ 
so that
\begin{equation}
\frac{a(2d-1)}{r} = 1-k\varepsilon.
\label{adr}
\end{equation}
Here, $K$ is chosen such that it is the maximal integer for which
\begin{equation} 
\beta_K \leq 
\begin{cases}
  \frac{\log(\alpha \log T)}{\alpha \log T} & \mbox{ if } \alpha \log T \to \infty, \\ 
  c & \mbox{ if } \alpha \ll \frac{1}{\log T},
  \end{cases}
\label{beta_k}
\end{equation}
where $c>0$ is a small constant such that
 \begin{equation} c^{1-d} \Big( \frac{a^d r^{1-d}}{r^{1-d}-1}+\frac{2r}{r-1} \Big) \leq 1-a.
 \label{condition_c2}
 \end{equation} Note that the above ensures that the conditions in Propositions \ref{j+1} and \ref{K_lemma} are satisfied. 
\kommentar{\acom{Some thoughts about the size of $\beta_K$: 
In terms of size, we need roughly 
\begin{equation} 
\beta_K^{1-d} \asymp (1-a) (r^{1-d}-1).\label{bksize}
\end{equation}
Let
$$ \frac{a(2d-1)}{r} = 1- \frac{n \log_3 T}{\log_2 T}.$$
We then have
$$a >1-  \frac{n \log_3 T}{\log_2 T},$$ so
$$1-a<  \frac{n \log_3 T}{\log_2 T}.$$
We also have
$$d > 1-   \frac{n \log_3 T}{2\log_2 T}$$
so
$$1-d<  \frac{n \log_3 T}{2\log_2 T}.$$ Also
$$r< \frac{1}{1-  \frac{n \log_3 T}{2\log_2 T}},$$ so 
$$r-1< \frac{  \frac{n \log_3 T}{2\log_2 T}}{1- \frac{n \log_3 T}{2\log_2 T}}.$$ Then
$$r^{1-d}-1 = \exp ( (1-d) \log(1+ r-1)) -1 \leq \exp  \Big( \frac{
(nx)^2}{4(1-nx/2)} \Big)-1 \sim \frac{(nx)^2}{4(1-nx/2)},$$ where
$x = \log_3 T/ \log_2T$.
From \eqref{bksize}, we have
$$ (1-d) \log \beta_K \sim \log(1-a)+\log(r^{1-d}-1).$$
Using all of the above, we have 
$$ nx/2 \log \beta_K \sim \log x + 2 \log x - \log (1-nx/2) \sim  nx/2+3 \log x,$$
so 
$$\beta_K \leq \frac{ \exp  \Big(\frac{6 \log_4 T \log_2 T}{n\log_3 T} \Big)}{(\log T)^{6/n}} .$$
  }}
  
If $t \notin \mathcal{T}_0$, then there exists $0 \leq v \leq K$ such that $\frac{ke^2}{\ell_0}|P_{0,v}(t)|>1$. Then 
	we have
	\begin{align*}
		\int_{[T,2T] \setminus \mathcal{T}_0} | \zeta(\tfrac12+\alpha+it)|^{-2k} \, dt& \leq \int_T^{2T} \Big( \frac{ke^2}{\ell_0} |P_{0,v}(t)| \Big)^{2s_0} | \zeta(\tfrac12+\alpha+it)|^{-2k}   \, dt \nonumber \\
		& \leq \Big( \frac{ke^2}{\ell_0} \Big)^{2s_0}\Big( \frac{1}{1-(\log T)^{-2\alpha}} \Big)^{ \frac{(1+\varepsilon) k\log T}{\log \log T}}\int_T^{2T} |P_{0,v}(t)|^{2s_0} \, dt,
	\end{align*}
by using the pointwise bound in Lemma \ref{CClemma},
	\begin{equation}\label{CCbound}
|\zeta(\tfrac12+\alpha+it)|^{-1} \ll \Big( \frac{1}{1-(\log T)^{-2\alpha}} \Big)^{ \frac{(1+\varepsilon) \log T}{2 \log \log T}}.
\end{equation}
By Proposition \ref{contrib_0} and Stirling's formula, we get that
\begin{align}
	&\int_{[T,2T] \setminus \mathcal{T}_0} |\zeta(\tfrac12+\alpha+it)|^{-2k} \, dt \ll   T \Big( \frac{1}{1-(\log T)^{-2\alpha}} \Big)^{ \frac{(1+\varepsilon)k \log T}{ \log \log T}} \nonumber \\
	& \qquad\times  \sqrt{s_0}\exp\bigg(-(2d-1)s_0 \log s_0+2s_0 \log \Big(ke^{3/2} b(\Delta_0) \Big(\log \frac{1}{\Delta_0 \alpha}\Big)^{\gamma(\Delta_0)} \sqrt{\log  \log T^{\beta_0}}  \Big) \bigg)\nonumber\\
&\quad=o(T),\label{0}
\end{align}
using the choice of $s_0$ in \eqref{beta_0}.

Now assume that $t \in \mathcal{T}_0$. Using Lemma \ref{lem_initial}, we have that 
$$ \int_{\mathcal{T}_0} |\zeta(\tfrac12+\alpha+it)|^{-2k} \, dt \leq \int_T^{2T} S_1(t) \, dt + \int_T^{2T} S_2(t) \, dt.$$
With the choice \eqref{beta_k} of $\beta_K$, we have that
$$ \Big(\frac{1}{1-T^{-\beta_K \alpha}} \Big)^{\frac{2k}{\beta_K }}\ll \\
\begin{cases}
1 & \mbox{ if }  \alpha \gg  \frac{1}{\log T},\\
(\log T)^{O(1)} & \mbox{ if } \alpha = o \big(  \frac{1}{\log T} \big)
\end{cases}$$  and 
$$ \exp \Big( O \Big(  \frac{\Delta_K^2 e^{\pi \Delta_K}}{1+\Delta_K t} + \frac{\Delta_K \log (1+ \Delta_K \sqrt{t})}{\sqrt{t}}+1\Big)  \Big) = O(1),$$ so
\begin{align*}
\int_T^{2T}  S_1(t) \, dt& \ll  (\log \log T)^k  \exp \Big( 2k \Big( \log \frac{1}{\Delta_K \alpha} \Big)^{\gamma(\Delta_K)} \Big) ( \log T)^{c_1 \gamma(\Delta_K)}\\
&\qquad \times  \int_T^{2T} \prod_{h=0}^K \max \Big\{1, |E_{\ell_h}(kP_{h,K}(t))|^2\Big(1+\frac{1}{15e^{\ell_h}}\Big)^2 \Big\}  \, dt.
\end{align*}

Note that in the inequality above, we can assume without loss of generality that 

$$\max \Big\{1, |E_{\ell_h}(kP_{h,K}(t))|^2\Big(1+\frac{1}{15e^{\ell_h}}\Big)^2 \Big\}=  |E_{\ell_h}(kP_{h,K}(t))|^2\Big(1+\frac{1}{15e^{\ell_h}}\Big)^2.$$  Using Proposition \ref{K_lemma} and the observation that $\gamma(\Delta_K)=0$ in the first two cases, and $\gamma(\Delta_K)=1$ in the third case, we get that
\begin{equation}
\int_T^{2T}  S_1(t) \, dt \ll
\begin{cases}
 T (\log \log T)^k \big( \frac{\log (\alpha \log T)}{\alpha} \big)^{k^2} & \mbox{ if } \alpha \log T \to \infty, \\ 
 T (\log \log T)^k (\log T)^{k^2 (\frac{1}{1-T^{-c\alpha}})^2} & \mbox{ if } \alpha \asymp \frac{1}{\log T},\\
 T (\log \log T)^k \exp \Big( k^2\big( \frac{1}{4}+\varepsilon \big) (\log \log T) \big( \log \frac{1}{\alpha \log T} \big)^2 \Big)&\mbox{ if } \alpha = o \big( \frac{1}{\log T} \big).
\end{cases}
\label{s1}
\end{equation}


Now we will need to bound the contribution from $S_2(t)$. Using Proposition \ref{j+1} and Stirling's formula, we have
\begin{align*}
 \int_T^{2T}  S_2(t) \, dt&  \ll T (\log \log T)^k   \sum_{j=0}^{K-1} (K-j) \sqrt{s_{j+1}} \exp \bigg( \frac{2k}{\beta_j } \log \frac{1}{1-T^{-\beta_j \alpha}}-(2d-1) s_{j+1}\log s_{j+1} \\
&\qquad\qquad+ 2s_{j+1} \log \Big(ke^{3/2} b(\Delta_{j+1})\Big (\log \frac{1}{\Delta_{j+1} \alpha}\Big)^{\gamma(\Delta_{j+1})} \Big)+2k  \Big( \log \frac{1}{\Delta_j \alpha} \Big)^{\gamma(\Delta_j)} \bigg)\\
&\qquad\times  \big( \log T^{\beta_j} \big)^{k^2 b(\Delta_j)^2 (\log \frac{1}{\Delta_j \alpha})^{2 \gamma(\Delta_j)}}.
\end{align*}
In the equation above, note that if $\beta_j \alpha\log T \geq \varepsilon$, then 
$$\frac{2k}{\beta_j} \log \frac{1}{1-T^{-\beta_j\alpha}} \leq  \frac{2k}{\beta_j} \log \frac{1}{1-e^{-\varepsilon}}<(2d-1)  s_{j+1} \log s_{j+1},$$ so the contribution in this case will be $o(T(\log \log T)^k)$.
Now if $\beta_j\alpha \log T < \varepsilon$, then 
\[
\log  \frac{1}{1-T^{-\beta_j\alpha}}<\log\frac{1}{\alpha}-\log(\beta_j\log T)+\log\frac{1}{1-\varepsilon/2},
\]
so we obtain that
\begin{align}
&\int_T^{2T}  S_2(t) \, dt  \ll T (\log \log T)^k \sum_{j=0}^{K-1} (K-j) \sqrt{\frac{1}{\beta_{j+1}}} \exp \bigg( \frac{ \log \log T}{\beta_j} \Big( 2k \frac{ \log\frac{1}{\alpha}}{\log \log T}-\frac{a(2d-1)}{r}  \Big) \nonumber \\
&\ \ + \frac{\log(\beta_j \log T)}{\beta_j} \Big( \frac{a(2d-1)}{r}-2k \Big)+ \frac{2a}{r \beta_j} \log \Big(ke^{3/2} b(\Delta_{j+1}) \Big(\log \frac{1}{\Delta_{j+1} \alpha}\Big)^{\gamma(\Delta_{j+1})}\Big)  \nonumber\\\nonumber \\
&\  \ +\frac{2k}{\beta_j} \log \frac{1}{1-\varepsilon/2} +\frac{a(2d-1)}{r \beta_j}\log\frac{r}{a} + 2k \Big( \log \frac{1}{\Delta_j \alpha} \Big)^{\gamma(\Delta_j} \bigg) \big( \log T^{\beta_j} \big)^{k^2 b(\Delta_j)^2 (\log \frac{1}{\Delta_j \alpha})^{2 \gamma(\Delta_j)}}.\label{s2}
\end{align}

\section{Proof of Theorems \ref{mainthm1} and \ref{mainthm2} for ``big'' shifts $\alpha$}
\label{bigshift}
Here, we will prove the bound \eqref{1} in Theorem \ref{mainthm1} when $\alpha \gg \frac{1}{(\log T)^{\frac{1}{2k}-\varepsilon}}$, and the bounds \eqref{4}, \eqref{5'}, \eqref{6} in Theorem \ref{mainthm2}.
Recall that
$$ u = \frac{ \log\frac{1}{\alpha}}{\log \log T}.$$
The contribution from $t \notin \mathcal{T}_0$ and from $S_1(t)$ has already been bounded in Section \ref{proof} (see equations \eqref{0} and \eqref{s1}). Now we focus on bounding the contribution from $S_2(t)$.

First assume that $\alpha \gg \frac{1}{(\log T)^{\frac{1}{2k}-\varepsilon}}$ and $k \geq 1/2$.  
We rewrite \eqref{s2} as
\begin{align*}
\int_T^{2T}  S_2(t) \, dt & \ll T (\log \log T)^k \sum_{j=0}^{K-1} (K-j) \sqrt{\frac{1}{\beta_{j+1}}} \exp \bigg( \frac{ \log \log T}{\beta_j} \Big(2ku-\frac{a(2d-1)}{r}  \Big) \nonumber \\
&\qquad  + \frac{\log(\beta_j \log T)}{\beta_j} \Big( \frac{a(2d-1)}{r}-2k \Big)+ 2k \Big( \log \frac{1}{\Delta_j \alpha} \Big)^{\gamma(\Delta_j)} + O \Big( \frac{ \log \log \log T}{\beta_j} \Big) \bigg) \\
&\quad \times \big( \log T^{\beta_j} \big)^{k^2 b(\Delta_j)^2 (\log \frac{1}{\Delta_j \alpha})^{2 \gamma(\Delta_j)}}.
\end{align*}
Note that 
$$  \frac{a(2d-1)}{r}-2k < 0,$$ and 
using \eqref{adr} and the fact that $\alpha \gg \frac{1}{(\log T)^{\frac{1}{2k}-\varepsilon}}$, it follows that
$$ 2k u-\frac{a(2d-1)}{r}  \leq -k \varepsilon.$$
Hence we get that 
\begin{align}
\int_T^{2T}  S_2(t) \, dt & \ll T (\log \log T)^k \sum_{j=0}^{K-1} (K-j) \sqrt{\frac{1}{\beta_{j+1}}} \exp \bigg(- \frac{ k \varepsilon \log \log T}{\beta_j}+ 2k \Big( \log \frac{1}{\Delta_j \alpha} \Big)^{\gamma(\Delta_j)}\nonumber   \\
&\qquad\qquad + O \Big( \frac{ \log \log \log T}{\beta_j} \Big)  \bigg) \big( \log T^{\beta_j} \big)^{k^2 b(\Delta_j)^2 (\log \frac{1}{\Delta_j \alpha})^{2 \gamma(\Delta_j)}}. \label{s2'}
\end{align}

We first consider the contribution from those $j$ for which $\gamma(\Delta_j)=1$. Let $R_1$ denote this contribution. Using the fact that $K \ll \log \log T$ and after a relabeling of the $\varepsilon$, we have that
\begin{align*}
R_1   &\ll   T (\log \log T)^k \sum_j  \exp \bigg(- \frac{ k \varepsilon \log \log T}{\beta_j} +2k \log \frac{1}{\beta_j \alpha \log T}+ O \Big( \frac{ \log \log \log T}{\beta_j}\Big) \\
&\qquad\qquad+ \Big( \frac{1}{4}+\varepsilon \Big) k^2\Big(  \log\frac{1}{\beta_j \alpha \log T} \Big)^2 \log (\beta_j \log T)   \bigg),
\end{align*}  where the sum over $j$ is such that $\gamma(\Delta_j)=1$. 
Since $\alpha \gg \frac{1}{\log T}$, we have
$$  \Big(  \log\frac{1}{\beta_j \alpha \log T} \Big)^2 \log (\beta_j \log T) \ll \Big( \log \frac{1}{\beta_j} \Big)^2 \log \log T.$$
As $\gamma(\Delta_j)=1$, we have $\beta_j \to 0$ in the sum in $R_1$ above, and then it follows that 
\begin{equation}
R_1 = o \big( T(\log \log T)^k\big).
\label{r1}
\end{equation}

Now we consider the contribution in \eqref{s2'} from those $j$ with $\gamma(\Delta_j)=0$. Let $R_2$ denote that term. We have that
\begin{align*}
R_2 \ll T (\log \log T)^k \sum_j  \exp \bigg(- \frac{ k \varepsilon \log \log T}{\beta_j} + O \Big( \frac{ \log \log \log T}{\beta_j} \Big) + k^2 b(\Delta_j)^2 \log ( \beta_j \log T) \bigg),
\end{align*}
where the sum is over $j$ such that $\gamma(\Delta_j)=0$. Keeping in mind the choices for $\beta_j$ (equations \eqref{beta_j} and \eqref{beta_k}), it then follows that
\begin{equation}\label{r2}
R_2 =
\begin{cases}
o ( T (\log \log T)^k ) & \mbox{ if } \alpha \log T \to \infty, \\ 
o\Big( T (\log \log T)^k (\log T)^{k^2 (\frac{1}{1-T^{-c\alpha}})^2}\Big) & \mbox{ if } \alpha \asymp \frac{1}{\log T}.
\end{cases}
\end{equation}
Combining the bounds \eqref{0}, \eqref{s1}, \eqref{r1} and \eqref{r2}, the bound \eqref{1} follows when $\alpha \gg \frac{1}{(\log T)^{\frac{1}{2k}-\varepsilon}}$ and $k \geq 1/2$.

Now assume that $k<1/2$ and $\alpha \gg \frac{1}{\log T}$. If $\frac{a(2d-1)}{r} -2k \leq 0$, then the same argument as before works. Hence we assume we have $\frac{a(2d-1)}{r} -2k >0$. 
We rewrite the bound \eqref{s2} for $S_2(t)$ as 
\begin{align}
&\int_T^{2T}  S_2(t) \, dt  \ll T (\log \log T)^k \sum_{j=0}^{K-1} (K-j) \sqrt{\frac{1}{\beta_{j+1}}} \exp \bigg( \frac{ \log \log T}{\beta_j} \Big( - \frac{ 2k \log (\alpha \log T)}{\beta_j} \nonumber \\
&\qquad + \frac{\log \beta_j}{\beta_j} \Big( \frac{a(2d-1)}{r}-2k \Big)+ \frac{2}{ \beta_j} \Big( \log  \log \frac{1}{\Delta_{j+1} \alpha}\Big)^{\gamma(\Delta_{j+1})}+ 2k \Big( \log \frac{1}{\Delta_j \alpha} \Big)^{\gamma(\Delta_j)} \nonumber \\
&\qquad + k^2 b(\Delta_j)^2 \Big(\log \frac{1}{\Delta_j \alpha}\Big)^{2 \gamma(\Delta_j)} \log (\beta_j \log T) +O \Big( \frac{1}{\beta_j} \Big) \bigg).\label{s22}
\end{align}

In the equation above, we first consider the contribution from those $j$ for which $\gamma(\Delta_j)=1$, i.e., those $j$ for which $\beta_j = o (\frac{1}{\alpha \log T})$. As before, we denote this contribution by $R_1$.
We let
\begin{align*}
f(x)= - &\frac{ 2k \log (\alpha \log T)}{x} + \frac{\log x }{x} \Big( \frac{a(2d-1)}{r}-2k \Big)+ \frac{2}{x} \log  \log \frac{1}{x \alpha \log T}+ 2k  \log \frac{1}{x \alpha \log T}  \\
&+ k^2 \Big( \frac{1}{4}+\varepsilon\Big) \Big(\log \frac{1}{x \alpha \log T}\Big)^2 \log (x \log T) .
\end{align*}
By taking the derivative, we see that when $\frac{1}{\log T} \ll \alpha = o \big( \frac{\log \log T}{\log T} \big)$, the maximum of $f(x)$ (when $x = o (\frac{1}{\alpha \log T})$) is attained at some $x_0 \asymp \frac{1}{\log \log T}$, while if $\alpha \gg \frac{ \log \log T}{\log T}$, the function $f(x)$ is increasing on the interval under consideration. Hence, we get that 
 \begin{equation} R_1 =
 \label{r1_again}
 \begin{cases}
o(T (\log \log T)^k) & \mbox{ if } \alpha \gg \frac{ \log \log T}{\log T}, \\
O \Big(T \exp \Big(C_1 (\log \log T) \big( \log \frac{ \log \log T}{\alpha \log T} \big)^2 \Big) \Big)& \mbox{ if } \frac{1}{\log T} \ll \alpha = o \big(\frac{\log \log T}{\log T} \big),
\end{cases}
 \end{equation}
 for some $C_1>0$. 

Now we bound the contribution in \eqref{s22} from those $j$ for which $\gamma(\Delta_j)=0$. It is easy to see that in this case, the function in \eqref{s22} is decreasing in $j$, so
 $$R_2 = o( T(\log \log T)^k).$$
 Combining the above with \eqref{0}, \eqref{s1} and \eqref{r1_again}, the bounds \eqref{4} and \eqref{5'} follow.
 
Now we assume that $\frac{1}{(\log T)^{\frac{1}{2k}-\varepsilon}} \ll \alpha = o \big( \frac{1}{\log T}\big)$. We rewrite the bound \eqref{s2} for $S_2(t)$ as 
\begin{align*}
&\int_T^{2T}  S_2(t) \, dt  \ll T (\log \log T)^k \sum_{j=0}^{K-1} (K-j) \sqrt{\frac{1}{\beta_{j+1}}} \exp \bigg( - \frac{ 2k \log (\alpha \log T)}{\beta_j} \nonumber \\
&\ \ + \frac{\log \beta_j}{\beta_j} \Big( \frac{a(2d-1)}{r}-2k \Big)+ \frac{2}{ \beta_j}  \log  \log \frac{1}{\Delta_{j+1} \alpha}+ 2k \log \frac{1}{\Delta_j \alpha}  + O \Big(  (\log \log T)^3 + \frac{1}{\beta_j} \Big) \bigg).
\end{align*}
In the sum over $j$ above, the maximum is attained at $j_0$ such that $\beta_{j_0} \asymp (\alpha \log T)^{\frac{2k}{\frac{a(2d-1)}{r}-2k}}$. It then follows that
\begin{equation*}
\int_T^{2T} S_2(t) \, dt \ll T \exp \Big( C_2 \Big( \frac{1}{\alpha \log T} \Big)^{\frac{2k}{1 -2k-k\varepsilon}} \Big),
\end{equation*}
for some $C_2>0$ (and after a relabeling of the $\varepsilon$).
Combining the above and \eqref{0}, \eqref{s1}, the bound \eqref{6} follows when $\frac{1}{(\log T)^{\frac{1}{2k}-\varepsilon}} \ll \alpha = o ( \frac{1}{\log T} )$.

\section{Proof of Theorem \ref{mainthm1}, bound \eqref{1}; some recursive estimates}
\label{s_recursive}
Here, we will prove the bound \eqref{2}. To do that, we will use an inductive argument, which will be performed in Subsection \ref{recursive_section}. The first step of the argument in carried out in the next subsection.
\label{recursive}
\subsection{The range $\alpha \gg \frac{ (\log \log T)^{\frac{4}{k}+\varepsilon}}{(\log T)^{\frac{1}{2k}}}$, $k \geq 1/2$, the first step}
\label{subsection_first_improvement}
We previously obtained the bound \eqref{1} in the region $\alpha \gg \frac{1}{(\log T)^{\frac{1}{2k}-\varepsilon}}$ and $k \geq 1/2$. Hence,
we will assume that
\begin{equation}  \frac{ (\log \log T)^{\frac{4}{k}+\varepsilon}}{(\log T)^{\frac{1}{2k}}} \ll \alpha = o \Big(\frac{1}{(\log T)^{\frac{1}{2k}-\varepsilon}} \Big),
\label{range}
\end{equation}
when $k \geq 1/2$.

Let 
\begin{equation}
\alpha = \frac{ (\log \log T)^b}{(\log T)^{\frac{1}{2k}}},\label{choice_b}
\end{equation} where $b \geq 4/k+\varepsilon$. From \eqref{range}, we have that $b = o (\frac{\log \log T}{ \log \log \log T})$.
We will show that for any $\delta>0$, we have 
\begin{equation*}
\int_T^{2T} |\zeta(\tfrac12+\alpha+it)|^{-2k} \, dt \ll  T \exp \bigg( \frac{ (\log T)^{\frac{3(1+\delta)}{kb-1}}}{ \exp \big( \frac{2 \log \log T \log \log \log \log T}{ (kb-1) \log \log \log T} \big)}  \bigg).
\end{equation*}

We choose $\beta_0, \ell_0, s_0$ as in \eqref{beta_0}
and $\beta_j, \ell_j, s_j$ are chosen as in \eqref{beta_j}. We choose $a, d ,r$ such that
\begin{equation}
\frac{a(2d-1)}{r} = 1 - \frac{n \log \log \log T}{ \log \log T},
\label{adr2}
\end{equation}
where 
\begin{equation}
n = (2kb-2) \Big(1- \frac{10 \delta}{12(1+\delta)} \Big).
\label{choice_for_n}
\end{equation} For simplicity of notation, let $x = \frac{\log \log \log T}{\log \log T}$.
We can take
\begin{align}
a &=\frac{1-n(1-\frac{\delta}{24})x}{1-n(1-\frac{\delta}{12})x} , \qquad 
d = 1-\frac{n}{2}\Big(1-\frac{\delta}{12}\Big)x , \qquad 
r=\frac{ 1-n(1-\frac{\delta}{24})x}{1-nx}.
\label{adr3}
\end{align}
We choose $\beta_K$ such that 
$$ \beta_K^{1-d} \Big( \frac{a^d r^{1-d}}{r^{1-d}-1}+\frac{2r}{r-1}  \Big)\leq 1-a.$$ 
Again, the above inequality ensures that the conditions in Propositions \ref{j+1} and \ref{K_lemma} are satisfied. Note that the condition above can be re-expressed as
$$\beta_K^{1-d} \leq c_1 (1-a)(r-1)(1-d),$$
for some constant $c_1>0$. 
We then choose $K$ such that $\beta_K$ is the largest of the form in \eqref{beta_j} such that
\begin{equation}
\beta_K \leq \frac{ \exp \big( \frac{6 \log \log T \log \log \log \log T}{  n\log \log \log T} \big)}{(\log T)^{\frac{6}{n(1-\frac{\delta}{12})}}}.
\label{beta_k_new}
\end{equation}

If $t \notin \mathcal{T}_0$, then we proceed as in Section \ref{bigshift} and similarly to equation \eqref{0}, we get that
\begin{align}
	&\int_{[T,2T] \setminus \mathcal{T}_0} |\zeta(\tfrac12+\alpha+it)|^{-2k} \, dt \ll   T \Big( \frac{1}{1-(\log T)^{-2\alpha}} \Big)^{ \frac{(1+\varepsilon)k \log T}{ \log \log T}} \nonumber \\
	& \qquad\times  \sqrt{s_0}\exp\bigg(-(2d-1)s_0 \log s_0+2s_0 \log \Big(ke^{3/2} b(\Delta_0) \Big(\log \frac{1}{\Delta_0 \alpha}\Big)^{\gamma(\Delta_0)} \sqrt{\log  \log T^{\beta_0}}  \Big) \bigg)\nonumber\\
&\quad=o(T). \label{small_0}
\end{align}

Now we suppose that $t \in \mathcal{T}_0$. Similarly as in Section \ref{bigshift}, using Proposition \ref{K_lemma} we get that
\begin{align}
 \int_T^{2T} S_1(t) \, dt&  \ll  T(\log \log T)^k  \exp \Bigg( c_1 \frac{ (\log T)^{\frac{6}{n(1-\frac{\delta}{12})}} \log \log T}{   \exp \big( \frac{6 \log \log T \log \log \log \log T}{n(1-\frac{\delta}{12})\log \log \log T} \big)}   \Bigg)  \exp \Big( k^2 (\log \log T)^3 \Big)\nonumber \\
 & \ll T \exp \Bigg( \frac{ (\log T)^{\frac{6}{n(1-\frac{\delta}{12})}}}{ \exp \big( \frac{4 \log \log T \log \log \log \log T}{n \log \log \log T} \big)}  \Bigg), \label{s1_again}
\end{align} 
for some $c_1>0$ (note that the constant $c_1$ can change from line to line).  

To bound the contribution from $S_2(t)$, we proceed as in equation \eqref{s2} and obtain that
\begin{align}
&\int_T^{2T}  S_2(t) \, dt  \ll T (\log \log T)^k \sum_{j=0}^{K-1} (K-j) \sqrt{\frac{1}{\beta_{j+1}}} \exp \bigg( \frac{ \log \log T}{\beta_j} \Big( 2ku-\frac{a(2d-1)}{r}  \Big) \nonumber \\
&\ \ + \frac{\log(\beta_j \log T)}{\beta_j} \Big( \frac{a(2d-1)}{r}-2k \Big)+ \frac{2a}{r \beta_j} \log \Big(ke^{3/2} b(\Delta_{j+1}) \Big(\log \frac{1}{\Delta_{j+1} \alpha}\Big)^{\gamma(\Delta_{j+1})}\Big)  \nonumber\\\nonumber \\
&\  \ +\frac{2k}{\beta_j} \log \frac{1}{1-\varepsilon/2} +\frac{a(2d-1)}{r \beta_j} \log\frac{r}{a} + 2k \Big( \log \frac{1}{\Delta_j \alpha} \Big)^{\gamma(\Delta_j)} \bigg) \big( \log T^{\beta_j} \big)^{k^2 b(\Delta_j)^2 (\log \frac{1}{\Delta_j \alpha})^{2 \gamma(\Delta_j)}}.\label{again}
\end{align}
Since $\alpha = \frac{(\log \log T)^b}{(\log T)^{\frac{1}{2k}}}$ and given \eqref{adr2}, we get that 
\begin{align*}
\int_T^{2T}  S_2(t) \, dt &  \ll T (\log \log T)^k \sum_{j=0}^{K-1} \exp \bigg( \frac{ \log \log \log T}{\beta_j}(n-2kb+2) + (\log \log T)^3 + O \Big( \frac{1}{\beta_j} \Big) \bigg),
\end{align*}
where we used the fact that $K \ll \log \log T$ and the fact that
$$ \log \log \frac{1}{\beta_j \alpha \log T} \leq \log \log \frac{1}{\beta_0 \alpha \log T} \leq \frac{\log \log \log T}{2k} .$$
With the choice \eqref{choice_for_n}, we have that $n-2kb+2 \leq - \delta$, and then
\begin{align*}
\int_T^{2T} &  S_2(t) \, dt   \ll T (\log \log T)^k \sum_{j=0}^{K-1} \exp \bigg( -\frac{  \delta \log \log \log T}{\beta_j} + (\log \log T)^3+ O \Big( \frac{1}{\beta_j} \Big)  \bigg).
\end{align*}

In the sum over $j$ above, the maximum is attained at $j=K-1$, and given the choice \eqref{beta_k_new} for $\beta_K$, it follows that
\begin{equation}
\int_T^{2T} S_2(t) \, dt =o(T).
\label{s2_again} 
\end{equation}
 Combining equations \eqref{small_0}, \eqref{s1_again} and \eqref{s2_again}, it follows that
\begin{equation}
\int_T^{2T} |\zeta(\tfrac12+\alpha+it)|^{-2k} \, dt \ll  T \exp \Bigg( \frac{ (\log T)^{\frac{3(1+\delta)}{kb-1}}}{ \exp \big( \frac{2 \log \log T \log \log \log \log T}{ (kb-1) \log \log \log T} \big)}  \Bigg).
\label{first_step}
\end{equation}

\subsection{The range $\alpha \gg \frac{ (\log \log T)^{\frac{4}{k}+\varepsilon}}{(\log T)^{\frac{1}{2k}}}$, $k \geq 1/2$, a recursive bound}
\label{recursive_section}
Here, we have the same setup as in the previous subsection. Namely, we assume \eqref{range} and \eqref{choice_b}.

We will perform the same argument as before, but with a different choice of parameters. We suppose that at step $m-1$, for any $\delta>0$, we have the bound
\begin{equation}
\label{recursive}
\int_T^{2T}  |\zeta(\tfrac12+\alpha+it)|^{-2k} \, dt \ll T \exp \Bigg( \frac{ (\log T)^{(1+\delta) ( \frac{3}{kb-1})^{m-1}} \log \log T}{ \exp \big( \frac{2\cdot  3^{m-2} \log \log T \log \log \log \log T}{(kb-1)^{m-1} \log \log \log T} \big)}  \Bigg) \exp \Big( (1+\delta)k^2  (\log \log T)^3  \Big).
\end{equation}
Note that we proved the first step of the induction in Subsection \ref{subsection_first_improvement} (see equation \eqref{first_step}). 
Using \eqref{recursive}, we will show that
\begin{equation}
\label{recursive_to_prove}
\int_T^{2T}  |\zeta(\tfrac12+\alpha+it)|^{-2k} \, dt \ll T \exp \Bigg( \frac{ (\log T)^{(1+\delta) ( \frac{3}{kb-1})^{m}} \log \log T}{ \exp \big( \frac{2\cdot  3^{m-1} \log \log T \log \log \log \log T}{(kb-1)^{m} \log \log \log T} \big)}  \Bigg) \exp \Big(  (1+\delta)k^2 (\log \log T)^3  \Big).
\end{equation}
Let \begin{equation}
\varepsilon' = \frac{\delta(kb-1) \log \log T}{4kb(m-1)( \log \log T-2kb\log \log \log T+ \frac{ \delta(kb-1) \log \log \log T}{2(m-1)})},
\label{epsilon_3}
\end{equation}
and 

\begin{equation}
p = \frac{ \log\log T}{\log \log T - 2kb \varepsilon' \log \log \log T},\qquad q = \frac{ \log \log T}{2kb \varepsilon' \log \log \log T},
\label{p}
\end{equation}
so that $1/p+1/q=1$. Let
\begin{equation}
f = b(1-\varepsilon').
\label{f}
\end{equation}
We will perform the inductive argument as long as
\begin{equation}
\label{iterations}
 \Big( \frac{kpf-1}{3}\Big)^{m-1} < \frac{\delta \log \log T}{9 \log \log \log T}.
\end{equation}
We choose 
\begin{equation}
\beta_0 = \frac{3(2d-1) \big(1+\frac{\delta}{3}\big) (  \frac{3}{kpf-1})^{m-1} \exp \big( \frac{2\cdot  3^{m-2} \log \log T \log \log \log \log T}{(kpf-1)^{m-1} \log \log \log T} \big)}{4q (\log T)^{(1+\frac{\delta}{3})( \frac{3}{kpf-1})^{m-1}}}, \, s_0 =  \Big[ \frac{1}{q\beta_0} \Big],\, \ell_0 = 2 \Big\lceil \frac{s_0^d}{2}\Big\rceil,
\label{beta0_yet_again}
\end{equation}
where
\begin{align*}
a &=\frac{1-n(1-\frac{\delta }{12})x}{1-n(1-\frac{\delta}{6})x} , \qquad 
d = 1-\frac{n}{2}\Big(1-\frac{\delta}{6}\Big)x , \qquad
r=\frac{ 1-n\big(1-\frac{\delta}{12}\big)x}{1-nx}.
\end{align*}
As before, we have
$$\frac{a(2d-1)}{r} = 1- nx.$$
Recall that $x= \frac{\log \log \log T}{ \log \log T}$ and we choose
\begin{equation}
n = \frac{2 (kb-1)(kpf-1)^{m-1}\big(1-\frac{\delta}{24}\big)}{3^{m-1} \big(1+\frac{\delta}{3}\big)}.
\label{choice_n}
\end{equation}
We also choose $\beta_K$ such that $K$ is the maximal integer for which 
\begin{equation}
\beta_K \leq \frac{ \exp \big( \frac{6 \log \log T \log \log \log \log T}{  n \log \log \log T} \big)}{(\log T)^{\frac{6}{n(1-\frac{\delta}{24})}}}.
\label{beta_k_new_again}
\end{equation}

If $t \notin \mathcal{T}_0$, then there exists $0 \leq v \leq K$ such that $\frac{ke^2}{\ell_0}|P_{0,v}(t)|>1$. Using H\"{o}lder's inequality, we have
	\begin{align*}
		\int_{[T,2T] \setminus \mathcal{T}_0} & | \zeta(\tfrac12+\alpha+it)|^{-2k} \, dt \leq \int_T^{2T} \Big( \frac{ke^2}{\ell_0} |P_{0,v}(t)| \Big)^{2s_0} | \zeta(\tfrac12+\alpha+it)|^{-2k}   \, dt \nonumber\\
		&  \leq  \Big( \frac{ke^2}{\ell_0} \Big)^{2s_0}\Big( \int_T^{2T}  | \zeta(\tfrac12+\alpha+it)|^{-2kp} \, dt  \Big)^{\frac{1}{p}} \Big( \int_T^{2T} |P_{0,v}(t)|^{2s_0q} \, dt  \Big)^{\frac{1}{q}}.
\end{align*}
For the first integral above, we will use the bound \eqref{recursive} with $\delta \mapsto \delta/3$. Note that
$$\alpha = \frac{ (\log \log T)^f}{(\log T)^{\frac{1}{2kp}}}.$$
Then using the recursive bound \eqref{recursive}, we obtain that
\begin{align} \label{apriori}
\int_T^{2T}  | \zeta(\tfrac12+\alpha+it)|^{-2kp}& \ll T \exp \Bigg( \frac{ (\log T)^{(1+\frac{\delta}{3}) ( \frac{3}{kpf-1})^{m-1}} \log \log T}{ \exp \big( \frac{2\cdot 3^{m-2} \log \log T \log \log \log \log T}{(kpf-1)^{m-1} \log \log \log T} \big)}  \Bigg) \nonumber\\
&\qquad\qquad\times\exp \Big( \Big(1+\frac{\delta}{3}\Big) p^2 k^2 (\log \log T)^3 \Big).
\end{align}

Now using Proposition \eqref{contrib_0}, we have that 
\begin{align*}
\int_T^{2T} |P_{0,v}(t)|^{2s_0q}  \, dt \ll T (s_0 q)! b(\Delta_0)^{2s_0q} \Big( \log \frac{1}{\Delta_0 \alpha} \Big)^{2 s_0 q \gamma(\Delta_0)} (\log \log T^{\beta_0})^{s_0 q}.
\end{align*}
Combining the bound above and \eqref{apriori} and using Stirling's formula, we get that
\begin{align}
\int_{[T,2T] \setminus \mathcal{T}_0} & | \zeta(\tfrac12+\alpha+it)|^{-2k} \, dt \ll T \exp \Bigg( \frac{ (\log T)^{(1+\frac{\delta}{3}) ( \frac{3}{kpf-1})^{m-1}} \log \log T}{ \exp \big( \frac{2\cdot 3^{m-2} \log \log T \log \log \log \log T}{ (kpf-1)^{m-1} \log \log \log T} \big)}  \Bigg) \nonumber  \\
& \times  \exp \Big( \Big(1+\frac{\delta}{3}\Big)  p k^2 (\log \log T)^3 \Big)    s_0^{\frac{1}{2q}} \exp \bigg(-(2d-1)s_0 \log s_0 \nonumber \\
&\qquad\qquad+2s_0 \log \Big(ke^{3/2}  \sqrt{q} b(\Delta_0) \Big(\log \frac{1}{\Delta_0 \alpha}\Big)^{\gamma(\Delta_0)} \sqrt{\log  \log T^{\beta_0}}  \Big) \bigg).  \label{tb9} 
\end{align}

Recall the choice \eqref{beta0_yet_again} for $s_0$. Note that we have
$$\log \log \frac{1}{\Delta_0 \alpha} \leq\log \log \log T,\qquad\log \log (\beta_0 \log T)< \log \log  \log T$$ and
$$\log q \leq \log \log \log T.$$
Using the three bounds above in \eqref{tb9}, it follows that
\begin{align*}
&\int_{[T,2T] \setminus \mathcal{T}_0} | \zeta(\tfrac12+\alpha+it)|^{-2k} \, dt \ll T \exp \Bigg( \frac{ (\log T)^{(1+\frac{\delta}{3}) ( \frac{3}{kpf-1})^{m-1}} \log \log T}{ \exp \big( \frac{2\cdot 3^{m-2} \log \log T \log \log \log \log T}{ (kpf-1)^{m-1} \log \log \log T} \big)}  \Bigg) \nonumber  \\
&\qquad \times  \exp \Big(   \Big(1+\frac{\delta}{3}\Big)  pk^2(\log \log T)^3 \Big)   \exp \Big(-(2d-1)s_0 \log s_0+ 4 s_0 \log \log \log T +O(s_0) \Big).
\end{align*}
By \eqref{iterations} we get
\begin{align}\label{t0_again}
\int_{[T,2T] \setminus \mathcal{T}_0}  | \zeta(\tfrac12+\alpha+it)|^{-2k} \, dt &\ll T \exp \Bigg(- \frac{ (\log T)^{(1+\frac{\delta}{3}) ( \frac{3}{kpf-1})^{m-1}} \log \log T}{4 \exp \big( \frac{2\cdot 3^{m-2} \log \log T \log \log \log \log T}{ (kpf-1)^{m-1} \log \log \log T} \big)}  \Bigg) \nonumber\\ 
&\qquad\qquad \times \exp \Big((1+\delta)k^2 (\log \log T)^3 \Big)\nonumber\\
&\ll T  \exp \Big((1+\delta)k^2 (\log \log T)^3 \Big). 
\end{align}

Now suppose that $t \in \mathcal{T}_0$. Using Proposition \ref{K_lemma} and proceeding as before, we get that
\begin{align}
& \int_T^{2T} S_1(t) \, dt \ll (\log \log T)^k \exp \bigg( c_1 (  \log T )^{\frac{3^m(1+\frac{\delta}{3})}{(kb-1)(kpf-1)^{m-1}(1-\frac{\delta}{24})^2}}  \log \log T\nonumber \\
&\qquad \times \exp \Big( - \frac{ 3^{m} \log \log T \log \log \log \log T}{(kb-1)(kpf-1)^{m-1}\log \log \log T} \Big) \bigg) \exp \Big((1+\delta)k^2 ( \log \log T)^3 \Big),
\label{s1_''}
\end{align}
for some $c_1>0$, and 
where we trivially bounded $\gamma(\Delta_K) \leq 1$. Now given the choice of parameters in  \eqref{epsilon_3}, \eqref{p}, \eqref{f}, we have
$$ kpf-1 = (kb-1) \Big(1- \frac{\delta}{4(m-1)} \Big)+ O \Big(  \frac{\log \log \log T}{ \log \log T} \Big)> (kb-1) \Big(1 - \frac{7 \delta}{24(m-1)} \Big).$$ 
Using this in \eqref{s1_''} leads to
\begin{align}
\int_T^{2T} & S_1(t) \, dt \ll  T \exp \Bigg( \frac{ (\log T)^{(1+\delta) ( \frac{3}{kb-1})^{m}} \log \log T}{ \exp \big( \frac{2\cdot  3^{m-1} \log \log T \log \log \log \log T}{(kb-1)^m \log \log \log T} \big)}  \Bigg) \exp \Big(  (1+\delta)k^2 (\log \log T)^3  \Big).\label{s1_recursive}
\end{align}

To bound the contribution from $S_2(t)$, we proceed as in \eqref{again}.  We rewrite 
\begin{align}
\int_T^{2T}  S_2(t) \, dt &  \ll T (\log \log T)^k \sum_{j=0}^{K-1} \exp \bigg( \frac{ \log \log \log T}{\beta_j}(n-2kb+2) -  \frac{n \log \log \log T \log(\beta_j \log T)}{\beta_j \log \log T} \nonumber \\
& +\log \log T +k^2 (\log \log T)^3+O \Big( \frac{1}{\beta_j} \Big) \bigg), \label{s2tbd}
\end{align}
where we used the fact that $k \geq 1/2$. 
Note that the maximum in the sum over $j$ is attained either at $j=0$ or $j=K-1$. Now given the choices \eqref{beta0_yet_again} and \eqref{choice_n}, the contribution from $j=0$ is 
\begin{align*}
\ll \exp \Big(\frac{\log \log \log T}{\beta_0} \Big( -\frac{\delta(kb-1)}{4}+\frac{n \log \big( \frac{kpf-1}{3}\big)^{m-1}}{\log\log T}\Big) +O \Big(  \frac{1}{\beta_0} + (\log \log T)^3\Big)\Big),
\end{align*}
and again using the choices \eqref{choice_n} and \eqref{iterations}, it follows that the contribution from $j=0$ is negligible.

For the contribution from $j=K-1$, proceeding similarly as in the bound for $S_1(t)$, it follows that
\begin{align*}
\int_T^{2T} & S_2(t) \, dt \ll  T \exp \Bigg( \frac{ (\log T)^{(1+\delta) ( \frac{3}{kb-1})^{m}} \log \log T}{ \exp \big( \frac{2\cdot  3^{m-1} \log \log T \log \log \log \log T}{(kb-1)^m \log \log \log T} \big)}  \Bigg) \exp \Big( (1+\delta)k^2  (\log \log T)^3  \Big).
\end{align*}
Combining the above, \eqref{t0_again} and \eqref{s1_recursive}, the induction conclusion \eqref{recursive_to_prove} follows. 

Now taking $m$ maximal as in \eqref{iterations}, we get that
\begin{equation}
\int_T^{2T} | \zeta(\tfrac12+\alpha+it)|^{-2k} \, dt \ll  T \exp \Big((1+\delta) k^2(\log \log T)^3 \Big) .\label{improve2}
\end{equation}

\subsection{The range $\alpha \gg \frac{ (\log \log T)^{\frac{4}{k}+\varepsilon}}{(\log T)^{\frac{1}{2k}}}$, $k \geq 1/2$, once more}
Here, we use the same setup as in Subsection \ref{recursive_section}. Once again, we assume \eqref{range} and \eqref{choice_b}.

We will improve the bound \eqref{improve2}. The proof is similar to the proof in the previous cases, so we will skip some of the details.

As before, let
$$  p = \frac{ \log\log T}{\log \log T - 2kb \delta \log \log \log T}$$
and
\begin{equation}
\label{q_again}
q = \frac{ \log \log T}{2kb \delta \log \log \log T},
\end{equation} so that $1/p+1/q=1$.
We choose 
\begin{equation}
\label{5}
\beta_0 = \frac{(6d-5) \log \log \log T}{(1+2\delta)pqk^2 (\log \log T)^3},\quad s_0= \Big[ \frac{1}{q \beta_0} \Big],\quad \ell_0 = 2\Big\lceil \frac{s_0^d}{2}\Big \rceil.
\end{equation}
We choose $\beta_j, s_j, \ell_j$ as in \eqref{beta_j}, and $a, d, r$ are such that
$$ \frac{a(2d-1)}{r} = 1-\delta.$$
We also choose $K$ maximal such that 

\begin{equation} \label{betak_ct}
\beta_K \leq c,
\end{equation} where $c$ is a small constant as in \eqref{condition_c2}. Note that the conditions in Propositions \ref{j+1} and \ref{K_lemma} are satisfied.

We now proceed as before. If $t \notin \mathcal{T}_0$, then as in Subsection \eqref{recursive_section}, we have that
\begin{align*}
		\int_{[T,2T] \setminus \mathcal{T}_0} & | \zeta(\tfrac12+\alpha+it)|^{-2k} \, dt  \leq  \Big( \frac{ke^2}{\ell_0} \Big)^{2s_0}\Big( \int_T^{2T}  | \zeta(\tfrac12+\alpha+it)|^{-2kp} \, dt  \Big)^{\frac{1}{p}} \Big( \int_T^{2T} |P_{0,v}(t)|^{2s_0q} \, dt  \Big)^{\frac{1}{q}} \nonumber  \\
		& \ll T  \exp \Big((1+\delta)pk^2 (\log \log T)^3 \Big)  \exp \bigg(-(2d-1)s_0 \log s_0 \nonumber \\
		&\qquad\qquad\qquad+2s_0 \log \Big(ke^{3/2}  \sqrt{q} b(\Delta_0) \Big(\log \frac{1}{\Delta_0 \alpha}\Big)^{\gamma(\Delta_0)} \sqrt{\log  \log T^{\beta_0}}  \Big) \bigg).
\end{align*}
Note that $\gamma(\Delta_0)=0$ with the choice of parameters \eqref{5}. We deduce that
\begin{align}
		\int_{[T,2T] \setminus \mathcal{T}_0}  | \zeta(\tfrac12+\alpha+it)|^{-2k} \, dt   &\ll T  \exp \Big((1+\delta)pk^2 (\log \log T)^3 \Big)\nonumber\\
&\qquad\times  \exp \Big(-(2d-1)s_0 \log s_0 + 2s_0 \log \log \log T +O(s_0) \Big) \nonumber \\
		&=o(T),\label{ot}
\end{align}
by \eqref{q_again} and the choice of parameters in \eqref{5}. Using the choice of $\beta_K$ in \eqref{betak_ct}, we also get that
\begin{equation}\label{s1bis}
\int_T^{2T} S_1(t) \, dt \ll T (\log \log T)^k (\log T)^{k^2}.
\end{equation}

To bound the contribution from $S_2(t)$, we rewrite equation \eqref{again} using the fact that $\gamma(\Delta_j)=0$ for all $j$ and that $K \ll \log \log T$. We have 
\begin{align*}
\int_T^{2T}  S_2(t) \, dt &  \ll T (\log \log T)^{k+1} \sum_{j=0}^{K-1} \exp \bigg( \frac{ \log \log T}{\beta_j} \Big( 1-2k - \frac{2kb \log \log \log T}{\log \log T} \Big)\\
&\qquad\qquad+ \frac{4 \log \log \log T}{\beta_j} \Big(2k - \frac{a(2d-1)}{r} \Big)+k^2 \log \log T  + O\Big( \frac{1}{\beta_j} \Big)\bigg).
\end{align*}
In the sum over $j$ above, note that if $k=1/2$, then the exponential term is bounded by 
$$\exp \Big(- \frac{c_1 \log \log \log T}{\beta_j} + k^2 \log \log T \Big), $$ for some $c_1>0$. If $k>1/2$, then the exponential term is bounded by
$$ \exp \Big(- \frac{c_1  \log \log T}{\beta_j} + k^2 \log \log T \Big), $$ for some $c_1>0$. 
In both cases, we obtain that
$$ \int_T^{2T}  S_2(t) \, dt = o \Big(T (\log \log T)^k (\log T)^{k^2} \Big).$$
Combining the above, \eqref{s1bis} and \eqref{ot}, the bound \eqref{1} follows.

\subsection{The range $ \frac{1}{(\log T)^{\frac{1}{2k}}} \ll \alpha =o\big( \frac{ (\log \log T)^{\frac{4}{k}+\varepsilon}}{(\log T)^{\frac{1}{2k}}}\big)$}
Here, we will prove the bound \eqref{2}. We will only sketch the proof, since it is similar to the proof in the previous cases. We choose the parameters $\beta_j, \ell_j, s_j$ as in \eqref{beta_0}, \eqref{beta_j}, and $a, d , r$ as in \eqref{adr2} and \eqref{adr3}, where 
\begin{equation}
n = \frac{6}{1-\delta}.
\label{n_again'}
\end{equation}
We also choose $\beta_K$ as in \eqref{beta_k_new}. We now proceed as in Subsection \ref{subsection_first_improvement}.

As in \eqref{small_0}, we have 
\begin{align}
	&\int_{[T,2T] \setminus \mathcal{T}_0} |\zeta(\tfrac12+\alpha+it)|^{-2k} \, dt =o(T). \label{small_5}
\end{align}
Also, as in \eqref{s1_again}, and keeping in mind the choice \eqref{n_again'} for $n$, we get that
\begin{align}
\int_T^{2T} S_1(t) \, dt \ll T \exp \Bigg( \frac{ \log T \log \log T}{\exp(\frac{2 \log \log T \log \log \log \log T}{3 \log \log \log T} ) } \Bigg). 
\label{s1_again5}
\end{align}

To bound the contribution from $S_2(t)$, we proceed as in equation \eqref{again}. Since $\alpha \gg \frac{1}{(\log T)^{\frac{1}{2k}}}$, we have
\begin{align*}
\int_T^{2T}  S_2(t) \, dt &  \ll T (\log \log T)^k \sum_{j=0}^{K-1} \exp \bigg( \frac{ \log \log \log T}{\beta_j}(n+2)  + (\log \log T)^3+O \Big( \frac{1}{\beta_j} \Big) \bigg).
\end{align*}
The maximum in the sum above is attained when $j=0$. So, keeping in mind the choice of $\beta_0$ in \eqref{beta_0}, we obtain that 
\begin{align}
\int_T^{2T} S_2(t) \, dt \ll T \exp \Big( \frac{(4+\delta)  \log T \log \log \log T}{\log \log T}  \Big),\label{s2_5}
\end{align}
after a relabeling of the $\delta$.
Combining \eqref{small_5}, \eqref{s1_again5} and \eqref{s2_5}, the conclusion follows in this case.

\section{Proof of Theorems \ref{mainthm1} and \ref{mainthm2} for small shifts $\alpha$}
\label{smallshift}
Here, we will prove the bounds \eqref{3} and \eqref{7} . 
The proof will be similar to the one in the previous subsection, but we choose the parameters differently. Recall that
$$ u = \frac{ \log \frac{1}{\alpha}}{\log \log T}.$$
Let $\delta>0$. We choose
\begin{equation}
\beta_0 = \frac{(2ku+2d-1- \frac{a(2d-1)}{r}) \log \log T }{(1+\delta)ku \log T}, \quad s_0 =\Big [\frac{1}{\beta_0}\Big], \quad\ell_0 =2\Big\lceil\frac{ s_0^d}{2}\Big \rceil,
\label{beta0_again}
\end{equation}
where we pick
$$a = \frac{1-3k\varepsilon}{1-2k\varepsilon}, \quad r= \frac{1}{1-2k\varepsilon}, \quad d= \frac{2-7k\varepsilon}{2(1-3k\varepsilon)},$$
so that
\begin{equation}
\frac{a(2d-1)}{r} = 1-4k\varepsilon.
\label{adr2'}
\end{equation}
We further pick $\beta_j,\ell_j, s_j$ as in \eqref{beta_j}. We choose $K$ to be maximal such that
\begin{equation}
\beta_K \leq c,
\label{betac}
\end{equation} for $c$ a small constant such that 
\begin{equation}
\label{condition_c}
 c^{1-d} \Big( \frac{ra^d}{r^{1-d}-1}+ \frac{2r}{r-1} \Big)\leq 1-a - \beta_0^{1-d}.
 \end{equation}
Note that the above ensures that the conditions in Propositions \ref{contrib_0}, \ref{j+1} and \ref{K_lemma} are satisfied. 

If $t \notin \mathcal{T}_0$, then we proceed as before, and similarly to equation \eqref{0} we get that 
\begin{align*}
	&\int_{[T,2T] \setminus \mathcal{T}_0} |\zeta(\tfrac12+\alpha+it)|^{-2k} \, dt \ll   T^{1+(1+\delta)ku} \nonumber \\
	&\qquad \times  \exp\bigg(-(2d-1)s_0 \log s_0+2s_0 \log \Big(ke^{3/2} b(\Delta_0) \Big(\log \frac{1}{\Delta_0 \alpha}\Big)^{\gamma(\Delta_0)} \sqrt{\log  \log T^{\beta_0}}  \Big) \bigg).
\end{align*}
Keeping in mind the choice of parameters \eqref{beta0_again}, we obtain that
\begin{equation}
\int_{[T,2T] \setminus \mathcal{T}_0} |\zeta(\tfrac12+\alpha+it)|^{-2k} \, dt \ll T^{1+(1+\delta) \frac{2ku-\frac{a(2d-1)}{r}}{2ku - \frac{a(2d-1)}{r}+2d-1}} \exp \Big(O \Big( \frac{\log T \log \log \log T}{\log \log T} \Big)\Big).
\label{nott0'}
\end{equation}

Now if $t \in \mathcal{T}_0$, then since $\gamma(\Delta_K)=1$, we use Proposition \ref{K_lemma} and the expression \eqref{betac} for $\beta_K$ as before to get that
\begin{equation}
\int_T^{2T} S_1(t) \, dt \ll T (\log T)^{O(1)} \exp \bigg( \Big( \frac{1}{4}+\varepsilon \Big)k^2 \Big( \log \frac{1}{\alpha} \Big)^2 \log \log T \bigg).
\label{s1'}
\end{equation}
Similarly as before (see equation \eqref{s2}), we have
\begin{align*}
\int_T^{2T}  S_2(t) \, dt & \ll T (\log \log T)^k \sum_{j=0}^{K-1} (K-j) \sqrt{\frac{1}{\beta_{j+1}}} \exp \bigg( \frac{ \log \log T}{\beta_j} \Big( 2k u-\frac{a(2d-1)}{r}  \Big) \nonumber \\
&\ + \frac{\log(\beta_j \log T)}{\beta_j} \Big( \frac{a(2d-1)}{r}-2k \Big)+ \frac{2a}{r \beta_j} \log \Big(e^{3/2} b(\Delta_{j+1}) \Big(\log \frac{1}{\Delta_{j+1} \alpha}\Big)^{\gamma(\Delta_{j+1})} \Big)  \nonumber \\
&\ +\frac{2k}{\beta_j} \log \frac{1}{1-\varepsilon/2} \bigg) \big( \log T^{\beta_j}\big)^{k^2 b(\Delta_j)^2 (\log \frac{1}{\Delta_j \alpha})^{2 \gamma(\Delta_j)}}.
\end{align*}
As $K \ll \log \log T$, we further write the above as 
\begin{align*}
\int_T^{2T} & S_2(t) \, dt  \ll T\sum_{j=0}^{K-1} \exp \bigg( \frac{ \log \log T}{\beta_j} \Big( 2k u-\frac{a(2d-1)}{r}  \Big)+ O \Big( \frac{\log T \log \log \log T}{\log \log T} \Big) \bigg)
\end{align*}

Since $\alpha = o\big(\frac{1}{(\log T)^{\frac{1}{2k}-\varepsilon}}\big),$ we have that 
$$ 2ku - \frac{a(2d-1)}{r} \geq 2k\varepsilon.$$
Hence the sum over $j$ above achieves its maximum when $j=0$. Using \eqref{beta0_again}, we obtain that
\begin{align}
\label{s2'}
\int_T^{2T} S_2(t) \, dt \ll  T^{1+(1+\delta) \frac{2ku-\frac{a(2d-1)}{r}}{2ku - \frac{a(2d-1)}{r}+2d-1}} \exp \Big(O \Big( \frac{\log T \log \log \log T}{\log \log T} \Big)\Big).
\end{align}
Combining equations \eqref{adr2'}, \eqref{nott0'}, \eqref{s1'} and \eqref{s2'} and after a relabeling of the $\varepsilon$,  the bounds \eqref{3} and \eqref{7} follow.

\section{The asymptotic formula}
\label{section_formula}
In this section we shall prove Theorem \ref{thm_asymptotic}. We have
\begin{align}\label{asymp1}
&\frac{1}{2\pi i}\int_{1-iT/2}^{1+iT/2}e^{z^2/2}X^z\frac{1}{\zeta(\tfrac12+\alpha+it+z)^{k}\zeta(\tfrac12+\alpha-it+z)^{k}}\frac{dz}{z}\\
&\qquad\qquad=\sum_{m,n= 1}^{\infty}\frac{\mu_k(m)\mu_k(n)}{(mn)^{1/2+\alpha}}\Big(\frac mn\Big)^{-it}W\Big(\frac{mn}{X}\Big),\nonumber
\end{align}
where
\begin{equation}\label{asymp3}
W(x)=\frac{1}{2\pi i}\int_{1-iT/2}^{1+iT/2}e^{z^2/2}x^{-z}\frac{dz}{z},
\end{equation}
by writing the zeta-functions in \eqref{asymp1} as Dirichlet series and integrating term-by-term. On the other hand, by deforming the line of integration, the left hand side of \eqref{asymp1} is equal to
\begin{align*}
&|\zeta(\tfrac12+\alpha+it)|^{-2k}+\frac{1}{2\pi i} \int_{-(1-\varepsilon)\alpha-iT/2}^{-(1-\varepsilon)\alpha+iT/2}e^{z^2/2}X^z\frac{1}{\zeta(\tfrac12+\alpha+it+z)^{k}\zeta(\tfrac12+\alpha-it+z)^{k}}\frac{dz}{z}\\
&\qquad\qquad+ O\Big(\frac{e^{-T^2/8}X}{T}\max_{\varepsilon\alpha\leq\sigma\leq 1+\alpha}\frac{1}{|\zeta(\tfrac12+\sigma+i(t\pm T/2))|^{k}|\zeta(\tfrac12+\sigma-i(t\mp T/2))|^{k}}\Big).
\end{align*}
For $t\in[T,2T]$, the $O$-term is
\begin{align*}
\ll \frac{e^{-T^2/8}X}{T}T^{\frac{k\log\frac{1}{\varepsilon\alpha}}{\log\log T}}\ll e^{-T^2/8}XT^{O_k(1)},
\end{align*}
by  Lemma \ref{CClemma}. Hence, integrating over $t\in[T,2T]$ we obtain that
\begin{align*}
&\int_{T}^{2T}|\zeta(\tfrac12+\alpha+it)|^{-2k}dt=\sum_{m,n= 1}^{\infty}\frac{\mu_k(m)\mu_k(n)}{(mn)^{1/2+\alpha}}W\Big(\frac{mn}{X}\Big)\int_{T}^{2T}\Big(\frac mn\Big)^{-it}dt\\
&\qquad\qquad-\frac{1}{2\pi i} \int_{-(1-\varepsilon)\alpha-iT/2}^{-(1-\varepsilon)\alpha+iT/2}e^{z^2/2}X^z\int_{T}^{2T}\frac{1}{\zeta(\tfrac12+\alpha+it+z)^{k}\zeta(\tfrac12+\alpha-it+z)^{k}}dt\frac{dz}{z}\\
&\qquad\qquad+O\big(e^{-T^2/8}XT^{O_k(1)}\big).\nonumber
\end{align*}
Furthermore, by the Cauchy-Schwarz inequality the second term above is
\begin{align*}
&\ll_\varepsilon \frac{X^{-(1-\varepsilon)\alpha}}{\alpha}\int_{-T/2}^{T/2}e^{-z^2/2}\bigg(\int_{T}^{2T}|\zeta(\tfrac12+\varepsilon\alpha+i(t+z))|^{-2k}dt\bigg)^{1/2}\nonumber\\
&\qquad\qquad\qquad\qquad\times\bigg(\int_{T}^{2T}|\zeta(\tfrac12+\varepsilon\alpha-i(t-z))|^{-2k}dt\bigg)^{1/2}dz\nonumber\\
&\ll_\varepsilon \frac{X^{-(1-\varepsilon)\alpha}}{\alpha}T(\log\log T)^k(\log T)^{k^2}\ll_\varepsilon TX^{-(1-\varepsilon)\alpha}(\log T)^{k^2+1},
\end{align*}
in view of \eqref{1} and \eqref{4} in Theorems \ref{mainthm1} and \ref{mainthm2}, and the fact that $\frac{1}{\alpha} \ll \min\Big\{\frac{(\log T)^{\frac{1}{2k}}}{(\log\log T)^{\frac4k+\varepsilon}},\frac{\log T}{\log\log T}\Big\}$. Thus,
\begin{align}\label{asymp2}
\int_{T}^{2T}|\zeta(\tfrac12+\alpha+it)|^{-2k}dt&=\sum_{m,n= 1}^{\infty}\frac{\mu_k(m)\mu_k(n)}{(mn)^{1/2+\alpha}}W\Big(\frac{mn}{X}\Big)\int_{T}^{2T}\Big(\frac mn\Big)^{-it}dt\\
&\qquad\qquad+O\big(TX^{-(1-\varepsilon)\alpha}(\log T)^{k^2+1}\big)+O\big(e^{-T^2/8}XT^{O_k(1)}\big).\nonumber
\end{align}

We next consider the contribution of off-diagonal terms $m\ne n$ on the right hand side of \eqref{asymp2}, which is
\begin{align*}
\ll\sum_{m\ne n}\frac{d_k(m)d_k(n)}{(mn)^{1/2+\alpha}|\log m/n|}\Big|W\Big(\frac{mn}{X}\Big)\Big|.
\end{align*}
We note from \eqref{asymp3} that $W(x)\ll x^{-1}$ trivially and
\[
W(x)=1+O(x)+O\Big(\frac{e^{-T^2/8}}{xT}\Big)
\]
if $x\leq 1$, by moving the contour to the $-1$-line, and so this contribution is bounded by
\begin{align*}%
&\ll \sum_{\substack{m\ne n\\mn\leq X}}\frac{d_k(m)d_k(n)}{(mn)^{1/2+\alpha}|\log m/n|}+\frac{e^{-T^2/8}X}{T}\sum_{\substack{m\ne n\\mn\leq X}}\frac{d_k(m)d_k(n)}{(mn)^{3/2+\alpha}|\log m/n|}\\
&\qquad\qquad+X\sum_{\substack{m\ne n\\mn> X}}\frac{d_k(m)d_k(n)}{(mn)^{3/2+\alpha}|\log m/n|}\nonumber\\
&=E_1+E_2,\nonumber
\end{align*}
say, where $E_1$ and $E_2$ denote the sums with $\frac{m}{n}\notin [1/2,2]$ and $\frac{m}{n}\in [1/2,2]$, respectively. With $E_1$, $|\log m/n|\gg 1$ and we get
\begin{align}\label{asymp6}
E_1&\ll \sum_{\substack{mn\leq X}}\frac{d_k(m)d_k(n)}{(mn)^{1/2+\alpha}}+\frac{e^{-T^2/8}X}{T}\sum_{\substack{mn\leq X}}\frac{d_k(m)d_k(n)}{(mn)^{3/2+\alpha}}+X\sum_{\substack{mn> X}}\frac{d_k(m)d_k(n)}{(mn)^{3/2+\alpha}}\nonumber\\
&\ll X^{1/2-\alpha}(\log X)^{2k-1}+\frac{e^{-T^2/8}X}{T}.
\end{align}
For $E_2$, we use the fact that
\begin{equation}\label{asymp4}
\frac{d_k(m)d_k(n)}{(mn)^{\sigma}}\ll \frac{d_k(m)^2}{m^{2\sigma}}+\frac{d_k(n)^2}{n^{2\sigma}},
\end{equation}
and we have
\begin{align}\label{asymp7}
E_2&\ll \sum_{m\leq\sqrt{2X}}\frac{d_k(m)^2}{m^{1+2\alpha}}\sum_{\substack{n\ne m\\m/2\leq n\leq 2m}}\frac{1}{|\log m/n|}+\frac{e^{-T^2/8}X}{T}\sum_{m\leq\sqrt{2X}}\frac{d_k(m)^2}{m^{3+2\alpha}}\sum_{\substack{n\ne m\\m/2\leq n\leq 2m}}\frac{1}{|\log m/n|}\nonumber\\
&\qquad\qquad+X\sum_{m>\sqrt{X/2}}\frac{d_k(m)^2}{m^{3+2\alpha}}\sum_{\substack{n\ne m\\m/2\leq n\leq 2m}}\frac{1}{|\log m/n|}\nonumber\\
&\ll \log X\sum_{m\leq\sqrt{2X}}\frac{d_k(m)^2}{m^{2\alpha}}+\frac{e^{-T^2/8}X\log X}{T}\sum_{m\leq\sqrt{2X}}\frac{d_k(m)^2}{m^{2+2\alpha}}+X\sum_{m>\sqrt{X/2}}\frac{d_k(m)^2\log m}{m^{2+2\alpha}}\nonumber\\
&\ll X^{1/2-\alpha}(\log X)^{k^2}+\frac{e^{-T^2/8}X\log X}{T}.
\end{align}

We are left with the contribution of the diagonal terms $m=n$ on the right hand side of \eqref{asymp2}. By \eqref{asymp3} this is
\begin{align*}
&T\sum_{n= 1}^{\infty}\frac{\mu_k(n)^2}{n^{1+2\alpha}}W\Big(\frac{n^2}{X}\Big)=\frac{T}{2\pi i}\int_{1-iT/2}^{1+iT/2}e^{z^2/2}X^{z}\sum_{n= 1}^{\infty}\frac{\mu_k(n)^2}{n^{1+2\alpha+2z}}\frac{dz}{z}\\
&\qquad=\frac{T}{2\pi i}\int_{1-iT/2}^{1+iT/2}e^{z^2/2}X^{z}\zeta(1+2\alpha+2z)^{k^2}\prod_{p}\bigg(1-\frac{1}{p^{1+2\alpha+2z}}\bigg)^{k^2}\bigg(1+\sum_{j=1}^\infty\frac{\mu_k(p^j)^2}{p^{(1+2\alpha+2z)j}}\bigg)\frac{dz}{z}.
\end{align*}
We move the contour to the $-(1-\varepsilon)\alpha$-line, crossing a simple pole at $z=0$. In doing so, we get that this is equal to
\begin{align*}
&T\zeta(1+2\alpha)^{k^2}\prod_{p}\bigg(1-\frac{1}{p^{1+2\alpha}}\bigg)^{k^2}\bigg(1+\sum_{j=1}^\infty\frac{\mu_k(p^j)^2}{p^{(1+2\alpha)j}}\bigg)\\
&\qquad\qquad+O\big(TX^{-(1-\varepsilon)\alpha}\alpha^{-(k^2+1)}\big)+ O\big(e^{-T^2/8}XT^{O_k(1)}\big).\nonumber
\end{align*}
Thus,
\begin{align*}
&\int_{T}^{2T}|\zeta(\tfrac12+\alpha+it)|^{-2k}dt=T\zeta(1+2\alpha)^{k^2}\prod_{p}\bigg(1-\frac{1}{p^{1+2\alpha}}\bigg)^{k^2}\bigg(1+\sum_{j=1}^\infty\frac{\mu_k(p^j)^2}{p^{(1+2\alpha)j}}\bigg)\\
&\qquad\qquad+O\big(TX^{-(1-\varepsilon)\alpha}(\log T)^{k^2+1}\big)+ O\big(e^{-T^2/8}XT^{O_k(1)}\log X\big)+O\big(X^{1/2-\alpha}(\log X)^{k^2}\big),
\end{align*}
by combining the above with \eqref{asymp2}, \eqref{asymp6} and \eqref{asymp7}.
We choose $X=T^2$, then the error terms above become $T^{1-2(1-\varepsilon) \alpha} (\log T)^{k^2+1}$, and the conclusion follows after a relabeling of the $\varepsilon$.


\section{Proof of Theorem \ref{mobius}}
\label{section_mobius}

\subsection{Assuming RH: $k\geq1$}

Following the arguments in \cite[Chapter 17]{D}, it is standard from Perron's formula that 
\begin{equation}\label{mobiusstart}
	\sum_{n \leq x} \mu_k(n)=\frac{1}{2\pi i}\int_{c-i[x]}^{c+i[x]}\frac{x^s}{\zeta(s)^k}\frac{ds}{s}+O\big((\log x)^k\big)
\end{equation}
with $c=1+\frac{1}{\log x}$. We deform the contour by replacing the line segment $c+it$, $|t|\leq [x]$, with a piecewise linear path comprising of a number
of horizontal and vertical line segments, $$\bigcup_{j=1}^{J} \big(V_j\cup H_j\big) \bigcup V_0,$$ 
where
\begin{align*}
	&V_0:\quad s=\frac{1}{2}+\frac{(\log\log x_0)^{\frac{8}{k}+\varepsilon}}{(\log x_0)^{\frac1k}}+it,\quad |t|\leq x_0,\\
	&V_j:\quad s=\frac{1}{2}+\frac{(\log \log (2^{j-1}x_0))^{\frac{8}{k}+\varepsilon}}{(\log (2^{j-1}x_0))^{\frac1k}}+it,\quad 2^{j-1}x_0\leq |t|\leq \min\{2^jx_0,[x]\},\quad 1\leq j\leq J,\\
	&H_j:\quad s=\sigma\pm i2^jx_0,\quad \frac12+\frac{(\log\log (2^{j}x_0))^{\frac{8}{k}+\varepsilon}}{(\log (2^{j}x_0))^{\frac1k}}\leq \sigma\leq \frac{1}{2}+\frac{(\log \log (2^{j-1}x_0))^{\frac{8}{k}+\varepsilon}}{(\log (2^{j-1}x_0))^{\frac1k}},\\
	&\qquad\qquad\qquad\qquad\qquad\qquad\qquad\qquad\qquad\qquad\qquad\qquad\qquad\qquad\qquad\qquad 1\leq j\leq J-1,\\
	& H_J:\quad s=\sigma\pm i[x],\quad \frac12+\frac{(\log\log (2^{J-1}x_0))^{\frac{8}{k}+\varepsilon}}{(\log (2^{J-1}x_0))^{\frac1k}}\leq \sigma\leq 1+\frac{1}{\log x}
\end{align*}
and $x_0=\exp\big((\log x)^{\frac{k}{k+1}}(\log\log x)^{\frac{k+8}{k+1}}\big)$, $J=\lceil \log_2\frac{[x]}{x_0} \rceil\ll\log x$. We encounter no pole in doing so.

We first consider the integral along $V_0$. Let $$\alpha_0:=\frac{(\log\log x_0)^{\frac{8}{k}+\varepsilon}}{(\log x_0)^{\frac1k}}\qquad\text{and}\qquad T_0:=\exp\Big(\frac{\log x_0}{(\log\log x_0)^{8}}\Big).$$ For $T_0\leq T\leq x_0$ 
we have $\frac{1}{(\log T)^{\frac1k}}\ll\alpha_0\ll \frac{(\log\log T)^{\frac{8}{k}+\varepsilon}}{(\log T)^{\frac1k}}$ 
and so
\begin{align*}
\frac{1}{2\pi i}\int_{1/2+\alpha_0+iT/2}^{1/2+\alpha_0+iT}\frac{x^s}{\zeta(s)^k}\frac{ds}{s}&\ll \frac{x^{1/2+\alpha_0}}{T}\int_{1/2+\alpha_0+iT/2}^{1/2+\alpha_0+iT}\Big|\frac{ds}{\zeta(s)^k}\Big|\\
&\ll x^{1/2+\alpha_0}\exp\Big(\log T(\log\log T)^{-1+\varepsilon}\Big),
\end{align*}
by \eqref{2}. We hence get
\begin{align}\label{0thestimate}
\frac{1}{2\pi i}\int_{1/2+\alpha_0+iT_0}^{1/2+\alpha_0+ix_0}\frac{x^s}{\zeta(s)^k}\frac{ds}{s}&\ll x^{1/2+\alpha_0}\exp\Big(\log x_0(\log\log x_0)^{-1+\varepsilon}\Big)\nonumber\\
&=\sqrt{x}\exp\Big(\frac{\log x(\log\log x_0)^{\frac{8}{k}+\varepsilon}}{(\log x_0)^{\frac1k}}+\log x_0(\log\log x_0)^{-1+\varepsilon}\Big)\nonumber\\
&\ll \sqrt{x}\exp\Big((\log x)^{\frac{k}{k+1}}(\log\log x)^{\frac{7}{k+1}+\varepsilon}\Big),
\end{align}
by diving the segment of integration  into dyadic intervals. The same bound holds for the integral along $\int_{1/2+\alpha_0-ix_0}^{1/2+\alpha_0-iT_0}$. Furthermore, we note from  Lemma \ref{CClemma} that
\begin{equation}\label{thepointwisebd}
	|\zeta(s)|^{-1} \ll (|t|+2)^{(\frac{1}{2k}+\varepsilon)\frac{\log\log x_0}{\log\log(|t|+2)}}
\end{equation}
for $s\in V_0$. So 
\begin{align}\label{1stestimate}
	\frac{1}{2\pi i}\int_{1/2+\alpha_0-iT_0}^{1/2+\alpha_0+iT_0}\frac{x^s}{\zeta(s)^k}\frac{ds}{s}&\ll \sqrt{x}\exp\Big(\frac{\log x(\log\log x_0)^{\frac{8}{k}+\varepsilon}}{(\log x_0)^{\frac1k}}+\frac{\log\log x_0\log T_0}{\log\log T_0}\Big)
\nonumber\\
	&\ll \sqrt{x}\exp\Big((\log x)^{\frac{k}{k+1}}(\log\log x)^{\frac{7}{k+1}+\varepsilon}\Big).
\end{align}
Combining \eqref{0thestimate} and \eqref{1stestimate} we obtain that
\begin{equation*}
\frac{1}{2\pi i}\int_{V_0}\frac{x^s}{\zeta(s)^k}\frac{ds}{s}\ll \sqrt{x}\exp\Big((\log x)^{\frac{k}{k+1}}(\log\log x)^{\frac{7}{k+1}+\varepsilon}\Big).
\end{equation*}

\kommentar{\hcom{I think we will need a bound like
\[
\frac{1}{T} \int_T^{2T} \powerzeta \, dt  \ll\exp\Big(\log T(\log\log T)^{-4+\varepsilon}\Big)
\]
uniformly for $\alpha\ll \frac{(\log\log T)^{8+\varepsilon}}{\log T}$. Assume that we have that bound, let's consider the integral along $V_0$. 

Denote by $\alpha_0=\frac{(\log\log x_0)^{8+\varepsilon}}{\log x_0}$. For $X\leq x_0$ we have $\alpha_0\ll \frac{(\log\log X)^{8+\varepsilon}}{\log X}$, and hence
\begin{align*}
\frac{1}{2\pi i}\int_{1/2+\alpha_0+iX/2}^{1/2+\alpha_0+iX}\frac{x^s}{\zeta(s)}\frac{ds}{s}&\ll \frac{x^{1/2+\alpha_0}}{X}\int_{1/2+\alpha_0+iX/2}^{1/2+\alpha_0+iX}\Big|\frac{ds}{\zeta(s)}\Big|\\
&\ll x^{1/2+\alpha_0}\exp\Big(\log X(\log\log X)^{-4+\varepsilon}\Big)\\
&\ll x^{1/2+\alpha_0}\exp\Big(\log x_0(\log\log x_0)^{-4+\varepsilon}\Big).
\end{align*}
Hence, by diving the vertical line $V_0$ into dyadic interval, we obtain
\begin{align*}
\frac{1}{2\pi i}\int_{V_0}\frac{x^s}{\zeta(s)}\frac{ds}{s}&\ll x^{1/2+\alpha_0}\exp\Big(\log x_0(\log\log x_0)^{-4+\varepsilon}\Big)\\
&=\sqrt{x}\exp\Big(\frac{\log x(\log\log x_0)^{8+\varepsilon}}{\log x_0}+\log x_0(\log\log x_0)^{-4+\varepsilon}\Big).
\end{align*}
Choosing $x_0=\exp(\sqrt{\log x}(\log\log x)^6)$ we get
\[
\frac{1}{2\pi i}\int_{V_0}\frac{x^s}{\zeta(s)}\frac{ds}{s}\ll \sqrt{x}\exp\Big(\sqrt{\log x}(\log\log x)^{2+\varepsilon}\Big).
\]
}}

For the contribution from the vertical segments $\cup_{j=1}^{J} V_j$, we deduce from \eqref{1} that it is bounded by
\begin{align}\label{estimateVj}
	&\ll \sqrt{x}(\log x)^{\frac {k^2}{4}+\varepsilon}\sum_{j=0}^{J-1}\exp\Big(\frac{\log x(\log\log (2^{j}x_0))^{\frac{8}{k}+\varepsilon}}{(\log (2^{j}x_0))^{\frac 1k}}\Big)\nonumber\\
	&\ll \sqrt{x}\exp\Big(\frac{\log x(\log\log x)^{\frac{8}{k}+\varepsilon}}{(\log x_0)^{\frac 1k}}\Big)\ll\sqrt{x}\exp\Big((\log x)^{\frac{k}{k+1}}(\log\log x)^{\frac{7}{k+1}+\varepsilon}\Big).
\end{align}

For $s\in H_j$, $1\leq j\leq J-1$, again like \eqref{thepointwisebd} we have
\[
|\zeta(s)|^{-1} \ll (2^jx_0)^{\frac{1}{2k}+\varepsilon}.
\]
So the contribution to \eqref{mobiusstart} from the horizontal segments $\cup_{j=1}^{J-1} H_j$ is
\begin{align}\label{estimateHj}
	&\ll \sqrt{x}\sum_{j=0}^{J-2}(2^jx_0)^{-1/2+\varepsilon}\exp\Big(\frac{\log x(\log\log (2^{j}x_0))^{\frac{8}{k}+\varepsilon}}{(\log (2^{j}x_0))^{\frac 1k}}\Big)\nonumber\\
&\ll \sqrt{x}x_0^{-1/2+\varepsilon}\exp\Big(\frac{\log x(\log\log x)^{\frac{8}{k}+\varepsilon}}{(\log x_0)^{\frac 1k}}\Big)\ll \sqrt{x}.
\end{align}

We are left with the integral along $H_J$, which is bounded by
\begin{equation}\label{HJ1}
\max_{\frac12+\frac{(\log\log x)^{\frac{8}{k}+\varepsilon}}{(\log x)^{\frac 1k}}\leq\sigma\leq1+\frac{1}{\log x}}\exp\Big((\sigma-1)\log x-k\log|\zeta(\sigma\pm ix)|\Big).
\end{equation}

For $\frac12+\frac{(\log\log x)^{\frac{8}{k}+\varepsilon}}{(\log x)^{\frac 1k}}\leq\sigma\leq \frac12+o(\frac{1}{\log\log x})$, Lemma \ref{CClemma} implies that this is
\begin{align}\label{est1}
	&\ll\exp\bigg((\sigma-1)\log x+\frac{k\log x}{2\log\log x}\log\frac{1}{1-(\log x)^{1-2\sigma}}\nonumber\\
&\qquad\qquad\qquad\qquad\qquad+O_k\Big(\frac{\log x}{\log\log x}\Big)+O_k\Big(\frac{\log x}{(\log\log x)^2}\log\frac{1}{1-(\log x)^{1-2\sigma}}\Big)\bigg)\nonumber\\
	&\ll  \exp\bigg((\sigma-1)\log x+\frac{k\log x}{2\log\log x}\log\frac{1}{(2\sigma-1)\log\log x}+O_k\big((2\sigma-1)\log x\big)\nonumber\\
&\qquad\qquad\qquad\qquad\qquad+O_k\Big(\frac{\log x}{\log\log x}\Big)+O_k\Big(\frac{\log x}{(\log\log x)^2}\log\frac{1}{2\sigma-1}\Big)\bigg),
\end{align}
which is decreasing with repsect to $\sigma$, and, hence, \begin{equation}\label{est2}\ll\exp\bigg(-\frac{(8+k)\log x\log\log\log x}{2\log\log x}+O_k\Big(\frac{\log x}{\log\log x}\Big)\bigg)\ll 1.\end{equation}

For $1-\frac{1}{\log\log x}\leq \sigma\leq 1+\frac{1}{\log x}$, the second estimate in  Lemma \ref{CClemma} leads to a bound of size
\begin{align}\label{est3}
	&\ll\exp\Big((\sigma-1)\log x+k\log\log\log x+O_k(1)\Big)\ll (\log\log x)^k.
\end{align}

Finally, for $\frac12+O(\frac{1}{\log\log x})\leq \sigma\leq 1-\frac{1}{\log\log x}$, we use the last estimate in  Lemma \ref{CClemma} to get the bound
\begin{align}\label{est4}
	&\ll\exp\bigg((\sigma-1)\log x+\frac{(\log x)^{2-2\sigma}}{(1-\sigma)\log\log x}+\varepsilon\log x+O_k\Big(\frac{(\log x)^{2-2\sigma}}{(1-\sigma)^2(\log\log x)^2}\Big)\bigg)\ll x^\varepsilon.
\end{align}
Combining the estimates we obtain the first part of the theorem for $k\geq1$.

\subsection{Assuming RH: $k<1$}

The arguments are similar to the previous subsection. 
We replace the contour in \eqref{mobiusstart} with $$\bigcup_{j=1}^{J} \big(V_j\cup H_j\big) \bigcup V_0,$$ 
where
\begin{align*}
	&V_0:\quad s=\frac{1}{2}+\varepsilon\frac{\log\log x_1}{\log x_1}+it,\quad |t|\leq x_1,\\
	&V_j:\quad s=\frac{1}{2}+\varepsilon\frac{\log \log (2^{j-1}x_1)}{\log (2^{j-1}x_1)}+it,\quad 2^{j-1}x_1\leq |t|\leq \min\{2^jx_1,[x]\},\quad 1\leq j\leq J,\\
	&H_j:\quad s=\sigma\pm i2^jx_1,\quad \frac12+\varepsilon\frac{\log\log (2^{j}x_1)}{\log (2^{j}x_1)}\leq \sigma\leq \frac{1}{2}+\varepsilon\frac{\log \log (2^{j-1}x_1)}{\log (2^{j-1}x_1)},\quad 1\leq j\leq J-1,\\
	& H_J:\quad s=\sigma\pm i[x],\quad \frac12+\varepsilon\frac{\log\log (2^{J-1}x_1)}{\log (2^{J-1}x_1)}\leq \sigma\leq 1+\frac{1}{\log x}
\end{align*}
and $x_1=\exp(\sqrt{\varepsilon\log x}\log\log x)$, $J=\lceil \log_2\frac{[x]}{x_1} \rceil\ll\log x$. Again we encounter no pole in doing so.

For the integral along $V_0$, let $$\alpha_1:=\varepsilon\frac{\log\log x_1}{\log x_1}\qquad\text{and}\qquad T_1:=\exp\Big(\frac{\log x_1}{\log\log x_1}\Big).$$ If $T_1\leq T\leq x_1$, then 
$\frac{1}{\log T}\ll\alpha_1\ll \frac{\log\log T}{\log T}$, 
and hence
\begin{align*}
\frac{1}{2\pi i}\int_{1/2+\alpha_1+iT/2}^{1/2+\alpha_1+iT}\frac{x^s}{\zeta(s)^k}\frac{ds}{s}&\ll \frac{x^{1/2+\alpha_1}}{T}\int_{1/2+\alpha_1+iT/2}^{1/2+\alpha_1+iT}\Big|\frac{ds}{\zeta(s)^k}\Big|\ll x^{1/2+\alpha_1}\exp\Big((\log\log T)^{1+\varepsilon}\Big),
\end{align*}
by \eqref{5'}. It follows that
\begin{align}\label{0thestimate'}
\frac{1}{2\pi i}\int_{1/2+\alpha_1+iT_1}^{1/2+\alpha_1+ix_1}\frac{x^s}{\zeta(s)^k}\frac{ds}{s}&\ll x^{1/2+\alpha_1}\exp\Big((\log\log x_1)^{1+\varepsilon}\Big)\nonumber\\
&=\sqrt{x}\exp\Big(\varepsilon\frac{\log x\log\log x_1}{\log x_1}+(\log\log x_1)^{1+\varepsilon}\Big)\nonumber\\
&\ll \sqrt{x}\exp\big(\varepsilon\sqrt{\log x}\big),
\end{align}
by relabelling $\varepsilon$.
The same bound holds for the integral along $\int_{1/2+\alpha_1-ix_1}^{1/2+\alpha_1-iT_1}$. Also, by \eqref{thepointwisebd} we have
\begin{align}\label{1stestimate'}
	\frac{1}{2\pi i}\int_{1/2+\alpha_1-iT_1}^{1/2+\alpha_1+iT_1}\frac{x^s}{\zeta(s)^k}\frac{ds}{s}&\ll \sqrt{x}\exp\Big(\varepsilon\frac{\log x\log\log x_1}{\log x_1}+\Big(\frac k2+\varepsilon\Big)\frac{\log\log x_1\log T_1}{\log\log T_1}\Big)
\nonumber\\
&\ll \sqrt{x}\exp\big(\varepsilon\sqrt{\log x}\big),
\end{align}
by another relabelling of $\varepsilon$. Combining \eqref{0thestimate'} and \eqref{1stestimate'} we obtain that
\begin{equation}\label{V0'}
\frac{1}{2\pi i}\int_{V_0}\frac{x^s}{\zeta(s)^k}\frac{ds}{s}\ll \sqrt{x}\exp\big(\varepsilon\sqrt{\log x}\big).
\end{equation}

The same bound holds for the other integrals. Indeed, similar to \eqref{estimateVj} and \eqref{estimateHj} we have
\begin{align}\label{Vj'}
\sum_{j=1}^J\frac{1}{2\pi i}\int_{V_j}\frac{x^s}{\zeta(s)^k}\frac{ds}{s}	&\ll \sqrt{x}(\log x)^{\frac {k^2}{4}+\varepsilon}\sum_{j=0}^{J-1}\exp\Big(\varepsilon\frac{\log x\log\log (2^{j}x_1)}{\log (2^{j}x_1)}\Big)\nonumber\\
	&\ll \sqrt{x}(\log x)^{\frac {k^2}{4}+1+\varepsilon}\exp\Big(\varepsilon\frac{\log x\log\log x_1}{\log x_1}\Big)\ll\sqrt{x}\exp\big(\varepsilon\sqrt{\log x}\big)
\end{align}
and
\begin{align}\label{Hj'}
\sum_{j=1}^{J-1}\frac{1}{2\pi i}\int_{H_j}\frac{x^s}{\zeta(s)^k}\frac{ds}{s}	&\ll \sqrt{x}\sum_{j=0}^{J-2}(2^jx_1)^{\frac k2-1+\varepsilon}\exp\Big(\varepsilon\frac{\log x\log\log (2^{j}x_1)}{\log (2^{j}x_1)}\Big)\nonumber\\
&\ll \sqrt{x}x_1^{-1/2+\varepsilon}(\log x)\exp\Big(\varepsilon\frac{\log x\log\log x_1}{\log x_1}\Big)\ll\sqrt{x}\exp\big(\varepsilon\sqrt{\log x}\big).
\end{align}
And similar to \eqref{HJ1}, \eqref{est1}, \eqref{est2}, \eqref{est3} and \eqref{est4} we get
\begin{align*}
&\frac{1}{2\pi i}\int_{H_J}\frac{x^s}{\zeta(s)^k}\frac{ds}{s}\ll \max_{\frac12+\varepsilon\frac{\log\log x}{\log x}\leq\sigma\leq1+\frac{1}{\log x}}\exp\Big((\sigma-1)\log x-k\log|\zeta(\sigma\pm ix)|\Big)\\
	&\ \ \ll  
\begin{cases}
\exp\Big(\frac{(k-1)\log x}{2}-\frac{k\log x\log\log\log x}{\log\log x}+O_k\big(\frac{\log x}{\log\log x}\big)\Big) & \text{if } \frac12+\varepsilon\frac{\log\log x}{\log x}\leq\sigma\leq \frac12+o(\frac{1}{\log\log x}),\\
(\log\log x)^k & \text{if } 1-\frac{1}{\log\log x}\leq \sigma\leq 1+\frac{1}{\log x},\\
x^\varepsilon & \text{if } \frac12+O(\frac{1}{\log\log x})\leq \sigma\leq 1-\frac{1}{\log\log x},
\end{cases}\\
&\ \ \ll x^\varepsilon,
\end{align*}
by Lemma  \ref{CClemma}. This, together with \eqref{V0'}, \eqref{Vj'} and \eqref{Hj'}, establishes the first part of the theorem for $k<1$.

\subsection{Assuming Conjecture \ref{gonek_conj}}

We replace the contour in \eqref{mobiusstart} by $H_1 \cup V_0$,
where
\begin{align*}
	&V_0:\quad s=\frac{1}{2}+\alpha_2+it,\quad |t|\leq [x],\\
&H_1:\quad s=\sigma\pm i[x],\quad \frac12+\alpha_2\leq \sigma\leq 1+\frac{1}{\log x}
\end{align*}for some $\frac{1}{\log x}\leq \alpha_2\leq 1$ to be chosen later.  
We encounter no pole in doing so.

On one hand, by Conjecture \ref{gonek_conj} we get
\begin{align}\label{new1stestimate}
	\frac{1}{2\pi i}\int_{V_0}\frac{x^s}{\zeta(s)^k}\frac{ds}{s}&\ll 
	\sqrt{x}x^{\alpha_2}\Big(\frac{1}{\alpha_2}\Big)^{\frac{k^2}{4}}.
\end{align}
On the other hand, we have
\[
\frac{1}{2\pi i}\int_{H_1}\frac{x^s}{\zeta(s)^k}\frac{ds}{s}\ll \max_{\frac12+\alpha_2\leq\sigma\leq1+\frac{1}{\log x}}\exp\Big((\sigma-1)\log x-k\log|\zeta(\sigma\pm i[x])|\Big).
\]
Like in \eqref{est1} and \eqref{est2}, this is
\begin{align}\label{newestH1}
&\ll  \exp\bigg((\sigma-1)\log x+\frac{k\log x}{2\log\log x}\log\frac{1}{(2\sigma-1)\log\log x}+O_k\big((2\sigma-1)\log x\big)\nonumber\\
&\qquad\qquad\qquad\qquad\qquad\qquad+O_k\Big(\frac{\log x}{\log\log x}\Big)+O_k\Big(\frac{\log x}{(\log\log x)^2}\log\frac{1}{2\sigma-1}\Big)\bigg)\nonumber\\
&\ll\exp\bigg(-\frac{\log x}{2}+\frac{k\log x\log\frac{1}{\alpha_2}}{2\log\log x}-\frac{k\log x\log\log\log x}{2\log\log x}\nonumber\\
&\qquad\qquad\qquad\qquad\qquad\qquad+O_k\big(\alpha_2\log x\big)+O_k\Big(\frac{\log x}{\log\log x}\Big)+O_k\Big(\frac{\log x\log\frac{1}{\alpha_2}}{(\log\log x)^2}\Big)\bigg)
\end{align}
for $\frac12+\alpha_2\leq\sigma\leq \frac12+o(\frac{1}{\log\log x})$,
  and as in \eqref{est3} and \eqref{est4}, it is
\[
\ll (\log\log x)^k+x^\varepsilon
\]
for $\frac12+O(\frac{1}{\log\log x})\leq \sigma\leq 1+\frac{1}{\log x}$. In view of \eqref{newestH1}, we choose
\[
\alpha_2=\begin{cases}
\frac{1}{\log x} & \text{if }k\leq2,\\
\frac{1}{(\log x)^{\frac{2}{k}}(\log\log x)^{1-\varepsilon}} & \text{if }k>2,
\end{cases}
\]
and then
\[
\frac{1}{2\pi i}\int_{H_1}\frac{x^s}{\zeta(s)^k}\frac{ds}{s}\ll\sqrt{x}.
\]
Combining with \eqref{new1stestimate} we obtain the second part of the theorem.
\bibliographystyle{amsalpha}

\end{document}